\documentclass[11pt]{article}
\usepackage{amssymb}
\usepackage{graphicx}
\usepackage{setspace}
  \usepackage{paralist}
  \usepackage{longtable}
   \usepackage{multirow}
    \usepackage{rotating}
\usepackage{marginnote}
\usepackage{amsmath, amsthm, amssymb}
\usepackage{authblk}
\usepackage{graphicx}
\usepackage{float}
\usepackage{hyperref}
\usepackage[margin=1in]{geometry}
\usepackage{comment}
\usepackage{color}
\specialcomment{details}{}

\allowdisplaybreaks[4]
\numberwithin{equation}{section}
\date{}

\excludecomment{details}

\newtheorem{theorem}{Theorem}[section]

\newtheorem{lemma}[theorem]{Lemma}
\newtheorem{corollary}[theorem]{Corollary}

\title{Singular solutions of linear problems\\ with fractional Laplacian}

\author{Congming Li$^{1,2}$
\ \ \ Chenkai Liu$^1$\ \ \ Zhigang Wu$^3$\ \ \  Hao Xu$^2$}

\begin{document}

\maketitle
\renewcommand{\thefootnote}{\fnsymbol{footnote}}

\footnotetext{1. School of Mathematical Sciences, MOE-LSC, Shanghai Jiao Tong University, China;}
\footnotetext{2. Department of Applied Mathematics, University of Colorado Boulder, USA;}
\footnotetext{3. Department of Applied Mathematics, Donghua University, China.}
\footnotetext{Partially supported by NSFC 11571233, NSF of Shanghai 16ZR1402100.}
\renewcommand{\thefootnote}{\arabic{footnote}}

\maketitle

\noindent{\bf{Abstract}}: In this paper, we study singular solutions of linear problems with fractional Laplacian.
First, we establish B\^{o}cher type theorems on a punctured ball via distributional approach.  Then, we develop a few interesting maximum principles on a punctured ball. Our distributional approach only requires the basic $L_{{\rm loc}}^1$-integrability. We also introduce several simple and useful lemmas, which enable us to unify the treatments for both Laplacian and fractional Laplacian. These theorems, lemmas and the methods introduced here can be adapted and applied in other situations.

\noindent{\bf{Keywords}}: fractional Laplacian, singular solution, B\^{o}cher theorem, maximum principle, regularities.

\noindent{\bf {MSC}}:   35B09, 
                        35B50, 
                        35S05.  

\pagestyle{myheadings} \thispagestyle{plain} 

\bigbreak

\section{Introduction}

\quad\!\!\quad The well-known B\^{o}cher theorem \cite{axler,bocher,raynor} for the nonnegative harmonic function states:\\
{\textbf{B\^{o}cher\ theorem}:
{\em If $u(x)\in L_{{\rm loc}}^1(B_1\backslash\{0\})$ and $u(x)$ is nonnegative and harmonic on $B_1\backslash\{0\}$, then there is a constant $a\geq0$
such that  $u(x)\in L_{{\rm loc}}^1(B_1)$ and
\begin{equation}\label{0.1}
-\Delta u(x)=a\delta_0(x),\ {\rm on}\ B_1
\end{equation}
where $\delta_0$ is the Delta distribution concentrated at the origin.}

The original proof is given by B\^{o}cher \cite{bocher} based on some non-obvious properties of the level surfaces of a
harmonic function.  Later, it was proved by Kellogg \cite{kellogg} using series expansions for spherical harmonics and
by Helms \cite{helms} via the potential theory and the theory of super-harmonic
functions. Recently, Axler \cite{axler} gave a simpler proof through the maximum principle, Harnack inequality and the solvability of the Dirichlet problem in a unit ball.

Br\'{e}zis-Lions \cite{brezis} obtained another B\^{o}cher type theorem:
\begin{theorem}
Let $u(x)\in L_{{\rm loc}}^1(B_1\backslash \{0\})$ be such that  $u(x)\geq0$ in $B_1$,
$\Delta u(x)\in L_{{\rm loc}}^1(B_1\backslash \{0\})$ and
\begin{equation}\label{0.1(000)}
-\Delta u(x)+Mu(x)\geq f(x),\ \  {\rm in}\ B_1,\ {\rm with}\ f\in L_{{\rm loc}}^1(B_1),\ M>0.
\end{equation}
Then $u(x)\in L_{{\rm loc}}^1(B_1)$ and there exists $\varphi(x)\in L_{{\rm loc}}^1(B_1)$ and $a\geq0$ such that
\begin{equation}\label{0.1(0000)}
-\Delta u(x)=\varphi(x)+a\delta_0,\ {\rm in}\ \mathcal{D}'(B_1).
\end{equation}
\end{theorem}
They rely on the assumption $f(x)\in L_{{\rm loc}}^1(B_1)$ and the sphere average method.

 In the last few decades, problems of fractional Laplacian type have attracted a lot of attention from scientists in both mathematical and physical sciences. However, due to the nonlocal nature of the fractional Laplacian, new methods are needed to derive B\^{o}cher type theorems for fractional Laplacian. Some simple cases are discussed in \cite{li0}. For a complete study, some unexpected difficulties arise. Quite a few subtle and interesting basic estimates are introduced to deal with these. These basic estimates are expected to have broader applications.
 The following are the generalized version of B\^{o}cher theorems for both the Laplacian and fractional Laplacian cases:

  \begin{theorem}[{\rm B\^{o}cher\ theorem\ for\ Laplacian}] \label{l 0.0}
Let $B_1\subset\mathbb{R}^n$ with $n\geq2$. Assume that $u(x)\in L_{{\rm loc}}^1(B_1\backslash \{0\})$ be a nonnegative function satisfying
\begin{equation}\label{0.00}
-\Delta u(x)+\vec b(x)\cdot\nabla u(x) + c(x)u(x)\geq0\ \ {\rm in}\ \mathcal{D}'(B_1\backslash \{0\}),
\end{equation}
where $\|\vec b(x)\|_{C^1(B_1)}+\|c(x)\|_{L^\infty(B_1)}\leq M$ for some constant $M$, then $u(x)\in L_{{\rm loc}}^1(B_1)$ and

\begin{equation}\label{0.000}
-\Delta u(x)+\vec b(x)\cdot\nabla u(x) + c(x)u(x)=\mu+a \delta_0(x)\ \ {\rm in}\ \mathcal{D}'(B_1),
\end{equation}
for a constant $a\geq0$ and a nonnegative Radon measure $\mu$ on $B_1$ satisfying $\mu(\{0\})=0$.
\end{theorem}

\begin{theorem}[{\rm B\^{o}cher\ theorem\ for\ fractional\ Laplacian}]\label{l 0.1}
Let $B_1\subset\mathbb{R}^n$ with $n\geq2$. Assume that $u(x)\in \mathcal{L}_{2s}$ with $s\in(\frac{1}{2},1)$  be a nonnegative function satisfying
\begin{equation}\label{0.0000}
(-\Delta)^s u(x)+\vec b(x)\cdot\nabla u(x) + c(x)u(x)\geq0\ \ {\rm in}\ \ \mathcal{D}'(B_1\backslash \{0\}),
\end{equation}
where $\|\vec b(x)\|_{C^1(B_1)}+\|c(x)\|_{L^\infty(B_1)}\leq M$ for some constant $M$,  then

\begin{equation}\label{eq111}
(-\Delta)^s u(x)+\vec b(x)\cdot\nabla u(x) + c(x)u(x)=\mu+a\delta_0(x)\ \ {\rm in}\ \mathcal{D}'(B_1),
\end{equation}
for a constant $a\geq0$ and a nonnegative Radon measure $\mu$ on $B_1$ satisfying $\mu(\{0\})=0$.

Furthermore, the theorem holds for $s\in(0,1)$ when $\vec{b}(x)\equiv0$ in (\ref{0.0000}).
\end{theorem}

 Inspired by the classical work of Berestycki-Nirenberg-Varadhan \cite{berestycki3}, Caffarelli-Nirenberg-Spruk \cite{caffarelli},  we derive the maximum principles for  `super-harmonic functions' and `fractional super-harmonic functions' on a punctured ball utilizing the above
 B\^{o}cher type theorems.
\begin{theorem}\label{l 0.2}
Let $B_r(x_0)\subset\mathbb{R}^n$ with $n\geq2$ and $r\leq1$, $\|\vec{b}(x)\|_{C^1(B_r(x_0))}<M$ and $c(x)\leq M$ in $B_r(x_0)$ for some constant $M>0$. Assume that $ u(x)\in L_{{\rm loc}}^1(B_r(x_0)\backslash\{x_0\})$ satisfies
\begin{equation}\label{6.1}
\left\{\ \begin{aligned}
&\!-\Delta u(x)+\vec{b}(x)\cdot\nabla u(x)+ c(x)u(x)\geq0,\ \ &{\rm in}\ \ &\mathcal{D}' (B_r(x_0)\backslash\{x_0\}),\\
&u(x)\geq m>0, &{\rm in}\ \  &B_r(x_0)\backslash B_{\frac{r}{2}}(x_0),\\
&u(x)\geq0, &{\rm in}\ \  &B_r(x_0).
\end{aligned}
\right.
\end{equation}
Then there exists a constant $\alpha >0$  depending on $n$ and $M$ only, such that $u(x)$ satisfies
\begin{equation}\label{6.2}
u(x)\geq \alpha m,\ \ {\rm in}\  B_r(x_0).
\end{equation}
In particular, when $M=0$, we have $\alpha=1$, i.e.
\begin{equation}\label{6.3}
u(x)\geq \alpha m,\ \ {\rm in}\  B_r(x_0).
\end{equation}
\end{theorem}

\begin{theorem} \label{l 1.2}
Let $B_r(x_0)\subset\mathbb{R}^n$ with $n\geq2$ and $r\leq1$, $\|\vec{b}(x)\|_{C^1(B_r(x_0))}\leq M$ and $c(x)\leq M$ in $B_r(x_0)$ for some constant $M>0$.
Assume that $ u(x)\in  \mathcal{L}_{2s}$ $(\frac{1}{2}<s<1)$ with $u(x)\geq0$, and satisfies
\begin{equation}\label{6.7}
\left\{\ \begin{aligned}
&(-\Delta)^su(x)+\vec{b}(x)\cdot\nabla u(x)+c(x)u(x)\geq0,\ &{\rm in}\ \ &\mathcal{D}'(B_r(x_0)\backslash\{x_0\}),\\
&u(x)\geq m>0,\ &{\rm in}\ \ &B_r(x_0)\backslash B_{\frac{r}{2}}(x_0),\\
&u(x)\geq0,\ &{\rm in}\ \ &\mathbb{R}^n.
\end{aligned}
\right.
\end{equation}
Then there exists a positive constant $\alpha$  depending on $n$, $s$ and $M$ only, such that
\begin{equation}\label{6.8}
u(x)\geq \alpha m,\ \ {\rm in}\  B_r(x_0).
\end{equation}
Moreover, the theorem holds for all $s\in(0,1)$ when $\vec{b}(x)\equiv0$ in (\ref{6.7}).
\end{theorem}

The following is another interesting  maximum principle:
\begin{theorem} \label{l 1.3}
Let $B_r(x_0)\subset\mathbb{R}^n$ with $n\geq2$ and $r\leq1$, $\|\vec{b}(x)\|_{C^1(B_r(x_0))}\leq M$ and $c(x)\leq M$ in $B_r(x_0)$ for some constant $M>0$.
Assume that $ u(x)\in  \mathcal{L}_{2s}$ $(0<s<1)$ with $u(x)\geq0$, and satisfies
\begin{equation}
\left\{\ \begin{aligned}
&(-\Delta)^su(x)+\vec{b}(x)\cdot\nabla u(x)+c(x)u(x)\geq0,\ &{\rm in}\ \ &\mathcal{D}'(B_r(x_0)),\\
&u(x)\geq m>0,\ &{\rm in}\ \ &B_r(x_0)\backslash B_{\frac{r}{2}}(x_0),\\
&u(x)\geq0,\ &{\rm in}\ \ &\mathbb{R}^n.
\end{aligned}
\right.
\end{equation}
Then there exists a positive constant $\alpha$  depending on $n$, $s$ and $M$ only, such that
\begin{equation}
u(x)\geq \alpha m,\ \ {\rm in}\  B_r(x_0).
\end{equation}
\end{theorem}
It is surprising that even when $s<\frac{1}{2}$, which means $(-\Delta)^s u$ is no longer a dominate term,  the maximum principle still holds.

These maximum principles are the basic tools to deal with singular solutions for both equations of Laplacian and fractional Laplacian. In particular, these are essential to study the equations of Laplacian or fractional Laplacian by Kelvin transform and methods of moving plane. See the related references \cite{chen2,chen3,chen4,chenc,jarohs} in this direction.
It is an interesting open problem to know if these maximum principles also hold in the anti-symmetric cases.

In proving these B\^{o}cher Type theorems and maximum principles, we develop some basic theorems in section \ref{5}, which are also interesting in their own.
The following is one of such theorems.
\begin{theorem}
Let $\Omega$ be a domain in $\mathbb{R}^n$. Assume $u(x), v(x) , f(x) ,g(x) \in L^1_{\rm loc}(\Omega) $, and satisfy
\begin{equation}
\begin{array}{rl}
&\displaystyle (-\Delta)^su(x)+\vec{b}(x)\cdot\nabla u(x)+c(x)u(x)\leq f(x)\ \ {\rm in}\ \mathcal{D}'(\Omega),\\
&\displaystyle (-\Delta)^sv(x)+\vec{b}(x)\cdot\nabla v(x)+c(x)v(x)\leq g(x)\ \ {\rm in}\ \mathcal{D}'(\Omega),
\end{array}
\end{equation}
where $\|\vec b(x)\|_{C^1(\Omega)}+\|c(x)\|_{L^\infty(\Omega)}<\infty$.  Then for $w(x)=\max\{u(x),v(x)\}$, it holds that
\begin{equation}
(-\Delta)^sw(x)+\vec{b}(x)\cdot\nabla w(x)+c(x)w(x)\leq f(x)\chi_{u> v} + g(x)\chi_{u<v}+\max\{f(x),g(x)\}\chi_{u=v}
\end{equation}
in the sense of distributions on $\Omega$.
\end{theorem}

Some simple cases of this lemma are well known with broad applications, our distributional approach makes it somewhat complete.
In fact, the  fractional Laplacian case is quite nontrivial.

We refer readers to some important and related work on the equations of the Laplacian or fractional Laplacian, \cite{berestycki1,berestycki2,gidas1,li} for the symmetry property,  \cite{berestycki3,cabre1} for maximum principles, \cite{caffarelli,chenc,korevaar} for singular solutions, and references therein.

Throughout this paper, the widely used mollification of $\Lambda\in\mathcal{D}'$ is $J_\delta \Lambda$:
\begin{equation}
(J_\delta \Lambda)(x)=(j^\delta\ast\Lambda)(x)=\Lambda(j^\delta_x),
\end{equation}
where $j(y)\in\mathcal{D}(B_1)$ is a positive smooth radially symmetric function supported in $B_1$ satisfying $\int_{\mathbb{R}^n}j(y)dy=1$. Moreover, $j^\delta(y)=\frac{1}{\delta}j(\frac{y}{\delta})$ and $j^\delta_x(y)=j^\delta(x-y)$.

The paper is organized as follows. In section \ref{2} we give a unified approach for the proofs of the B\^{o}cher type theorems for Laplacian
and fractional Laplacian, respectively. The maximum principles on punctured balls are established in section \ref{3}. In section \ref{4} and section \ref{5}, we provide some basic and interesting lemmas for fractional Laplacian, which are keys to derive the B\^{o}cher type theorems for the theorems.
%
%
%
%
\section{B\^{o}cher type theorems}\label{2}
\subsection{Proof of Theorem \ref{l 0.0}}
\quad\!\!\quad In this section, we prove Theorem \ref{l 0.0}.

%
\begin{proof}
We first show the existence of $\mu$ as a Radon measure on $B_1$ and that $u(x)\in L_{\rm loc}^1(B_1)$.

To begin with, we denote $\Lambda=-\Delta u(x)+\vec b(x)\cdot\nabla u(x) + c(x)u(x)$ as a distribution. $\Lambda$ is monotone nondecreasing by its nonnegativity. Thus we can extend it to  a nonnegative linear functional on $C_c(B_1\backslash\{0\})$. Applying the Riesz-Markov-Kakutani representation theorem, we can represent $\Lambda$  as
\begin{equation}\label{0.4}
\Lambda(\varphi)=\int_{B_1\backslash\{0\}}\varphi(x)d\mu,
\end{equation}
where $\mu$ is a nonnegative Radon measure defined on $B_1\backslash\{0\}$.

We now extend the definition of $\mu$ as a nonnegative Borel measure on the whole ball $B_1$ with $\mu(\{0\})=0$. Notice that $\mu(B_{r_0})$ may be infinity, since $B_{r_0}\backslash\{0\}$ is not a compact set in $B_1\backslash\{0\}$.
In order to show that $\mu$ is a nonnegative Radon measure on $B_1$, we need $\mu(B_{r}\backslash\{0\})<\infty$, for some $0<r<1$. From (\ref{0.4}) we have the inequality:
\begin{equation}\label{0.5}
-\Delta u(x)+{\rm div}(\vec b(x)u(x))+(2M+1)u(x)\geq u(x)+\mu\ \ {\rm in}\ \mathcal{D}'(B_1\backslash \{0\}).
\end{equation}

For $0<\alpha<1$, $0<\epsilon<1$ and $\sigma=\frac{2M+1}{n}$, define
\begin{equation}
\begin{array}{rl}
\varphi_\epsilon(r)=\bigg\{
\begin{array}{ll}
1+\sigma r^2-\big(\frac{\ln r}{\ln\epsilon}\big)^{\alpha},\ {\rm when}\ n=2,\\[2mm]
1+\sigma r^2-(\frac{\epsilon}{r})^{\alpha},\ {\rm when}\ n\geq3.
\end{array}
\end{array}
\end{equation}
Then we have: when $0<|x|\leq r_0$ for some $r_0$ independent of $\epsilon$,
\begin{equation}
-\Delta\varphi_\epsilon(|x|)-\vec b(x)\cdot\nabla\varphi_\epsilon(|x|)+(2M+1)\varphi_\epsilon(|x|)\leq 0.
\end{equation}
\begin{details}

Indeed,
\begin{equation}
\begin{array}{rl}
&\displaystyle -\Delta\varphi_\epsilon(|x|)-\vec b(x) \cdot\nabla\varphi_\epsilon(|x|)+(2M+1) \varphi_\epsilon(|x|)\\[2mm]
=\!\!&\displaystyle -\varphi_\epsilon''(|x|)-(\frac{n-1}{|x|}+\frac{\vec b(x)\cdot x}{|x|})\varphi_\epsilon'(|x|)+(2M+1) \varphi_\epsilon(|x|)\\[3mm]
=\!\!&[2M+1+(2M+1)\sigma |x|^2-2\sigma(n+\vec b(x)\cdot x)]\\
&-\displaystyle\bigg\{\begin{array}{ll}
\displaystyle \big(\frac{\ln|x|}{\ln\epsilon}\big)^{\alpha}[\frac{(1-\alpha)\alpha}{|x|^2(\ln |x|)^2}-\frac{\alpha\vec b(x)\cdot x}{|x|^2\ln|x|}+2M+1],\ {\rm when}\ n=2,\\[2mm]
\displaystyle\big(\frac{\epsilon}{|x|}\big)^{\alpha}[\alpha(n-\alpha-2)\frac{1}{|x|^2}+\alpha\frac{\vec b(x)\cdot x}{|x|^2}+2M+1],\ {\rm when}\ n\geq3.
\end{array}
\end{array}
\end{equation}

By choosing suitable $\sigma$  ($\sigma n=2M+1$) and $r_1$ ($r_1<\frac{1}{4M+1}$), we have, for $|x|<r_1$:
\begin{equation}
\begin{array}{rl}
&2M+1+(2M+1)\sigma |x|^2-2\sigma(n+\vec b(x)\cdot x)\\
\leq\!\!&-(2M+1)+2(2M+1)M|x|+(2M+1)^2 |x|^2<0.
\end{array}
\end{equation}

While for reasonable $r_2$ (only depends on $n$, $\alpha$ and $M$) and $|x|<r_2$, we have
\begin{equation}
\begin{array}{rl}
&\displaystyle\frac{(1-\alpha)\alpha}{|x|^2(\ln |x|)^2}-\frac{\alpha\vec b(x)\cdot x}{|x|^2\ln|x|}+2M+1\\
\geq\!\!&\displaystyle\frac{(1-\alpha)\alpha}{|x|^2(\ln |x|)^2}+\frac{\alpha M}{|x|\ln|x|}+2M+1>0,\ \ \ n=2
\end{array}
\end{equation}
or
\begin{equation}
\begin{array}{rl}
&\displaystyle\alpha(n-\alpha-2)\frac{1}{|x|^2}+\alpha\frac{\vec b(x)\cdot x}{|x|^2}+2M+1\\
\geq\!\!&\displaystyle\displaystyle\alpha(n-\alpha-2)\frac{1}{|x|^2}-\alpha\frac{M}{|x|}+2M+1>0,\ \ \ n\geq3.
\end{array}
\end{equation}

Hence, for $0<|x|<r_0:=\min\{r_1,r_2\}$, we have
\begin{equation}
-\Delta\varphi_\epsilon(|x|)-\vec b(x) \cdot\nabla\varphi_\epsilon(|x|)+(2M+1) \varphi_\epsilon(|x|)<0.
\end{equation}
\end{details}
We also define
\begin{gather*}
\psi_\epsilon(x)=\max\{0,\varphi_\epsilon(|x|)\},\qquad {\rm and}\qquad \psi_\epsilon^\delta(x)= J_\delta\psi_\epsilon(x),
\end{gather*}
 For this constructed, smooth function $\psi_\epsilon^\delta(x)$, by using a basic estimate we derive in theorem \ref{lckcor3}, one has for $|x|\leq r_0$:
\begin{equation}\label{0.5(00)}
-\Delta\psi_\epsilon^\delta(x)-\vec b(x) \cdot\nabla\psi_\epsilon^\delta(x)+(2M+1)\psi_\epsilon^\delta(x)\leq M\delta \|\nabla\varphi_\epsilon(x)\|_{L^\infty(\{\varphi_\epsilon>0\})}\chi_{\{\varphi_\epsilon>0\}}^\delta(x).
\end{equation}

In addition, we need to choose a cutoff function $\eta(x)\in C_0^\infty(B_1)$ satisfying $0\leq\eta(x)\leq1$,
\begin{equation}
\eta(x)= \left \{
\begin{array} {ll}
1, \ \  |x| \leq \frac{1}{2}r_0;\\[2mm]
0, \ \ |x| \geq \frac{3}{4}r_0.
\end{array}  \right.
\end{equation}
By letting $\delta$ sufficiently small such that $\eta(x)\psi_\epsilon^\delta(x)\in \mathcal{D}(B_1\backslash\{0\})$, we can have the following estimate after testing (\ref{0.5}) by the nonnegative function $\eta(x)\psi_\epsilon^\delta(x)$
\begin{equation}\label{0.6}
\begin{array}  {rl}
0\leq&\displaystyle\int_{B_1\backslash\{0\}}\eta(x)   \psi_\epsilon^\delta(x)u(x)dx+\int_{B_1\backslash\{0\}}\eta(x)   \psi_\epsilon^\delta(x)d\mu\\
=&\displaystyle\int_{B_1\backslash\{0\}} u(x) \{(-\Delta)(\eta(x)   \psi_\epsilon^\delta(x))-\vec b(x) \nabla(\eta(x)   \psi_\epsilon^\delta(x))+(2M+1)\eta(x)   \psi_\epsilon^\delta(x)\}dx \\
=&\displaystyle\int_{B_1\backslash\{0\}} u(x)\eta(x)\{\underbrace{(-\Delta)\psi_\epsilon^\delta(x) -\vec b(x) \cdot\nabla\psi_\epsilon^\delta(x)+(2M+1) \psi_\epsilon^\delta(x)}_{{\rm by \ using}\ (\ref{0.5(00)})}\}dx \\[3.5mm]
&+\displaystyle\int_{B_1\backslash\{0\}} u(x)\underbrace{\{\psi_\epsilon^\delta(x)(-\Delta)\eta(x)-2\nabla\eta(x)\cdot\nabla\psi_\epsilon^\delta(x)-\psi_\epsilon^\delta(x)\vec b(x)\cdot\nabla \eta(x) \}}_{{\rm only\ supported\ in}\ B_{r_0}\backslash B_{\frac{r_0}{4}}\ {\rm and\ bounded}}dx\\
\leq &\displaystyle\int_{B_{r_0}\backslash\{0\}} u(x)\eta(x)M\delta \|\nabla\varphi_\epsilon(x)\|_{L^\infty(\{\varphi_\epsilon>0\})}\chi_{\{\varphi_\epsilon>0\}}^\delta(x)dx+C_1\int_{B_{r_0}\backslash B_{\frac{r_0}{4}}}u(x)dx,
\end{array}
\end{equation}
where $C_1$ is independent of $\epsilon$.

For fixed $\epsilon>0$, letting $\delta\rightarrow0$ in (\ref{0.6}), we obtain
\begin{equation}\label{0.7}
\int_{B_1\backslash\{0\}}\eta(x)   \psi_\epsilon(x)u(x)dx+\int_{B_1\backslash\{0\}}\eta(x)   \psi_\epsilon(x)d\mu
\leq C_1\int_{B_{r_0}\backslash B_{\frac{r_0}{4}}}u(x)dx:=C_2.
\end{equation}

Noticing $\psi_\epsilon(x)\nearrow(1+\sigma |x|^2)$ as $\epsilon\rightarrow0$, we can apply monotone convergence theorem and derive
\begin{equation}\label{0.8}
\begin{array}{rl}
&\displaystyle\int_{B_\frac{r_0}{2}\backslash\{0\}}u(x)dx+\mu(B_\frac{r_0}{2}\backslash\{0\})\\
\leq\!\!&\displaystyle\int_{B_1\backslash\{0\}}\eta(x)(1+\sigma |x|^2)u(x)dx+\int_{B_1\backslash\{0\}}\eta(x)(1+\sigma |x|^2)d\mu\leq C_2.
\end{array}
\end{equation}
Then (\ref{0.8})  immediately implies $u(x)\in L_{{\rm loc}}(B_1)$. Hence $\mu$, as the extension of a nonnegative Radon measure on $B_1\backslash\{0\}$, is indeed a nonnegative Radon measure on the whole ball $B_1$.

Next, we prove $u(x)$ satisfies (\ref{0.000}) in three steps, under the assumption that $n\geq3$. The case of $n=2$ is similarly derived.

\textit{Step 1}. We claim that there exists constants $a$ and $\vec{d}$, independent of the test function $\varphi(x)$, such that
\begin{equation}\label{0.9}
-\Delta u(x)+\vec b(x)\cdot\nabla u(x)+c(x)u(x)= \mu+a\delta_0(x)+(\vec{d}\cdot\nabla)\delta_0(x),\ \ {\rm in}\ \mathcal{D}'(B_1).
\end{equation}

We first rewrite  (\ref{0.4}) as
\begin{equation}\label{0.10}
-\Delta u(x)+{\rm div}(\vec b(x)u(x))=[-c(x)+{\rm div}(\vec b(x))]u(x)+\mu:=\lambda\ \ {\rm in}\ \mathcal{D}'(B_1\backslash \{0\}),
\end{equation}
where $\lambda$ is a signed Radon measure on $B_1$.
Define $\eta(x)\in  \mathcal{D}(B_1)$ to be a function  satisfying $0\leq\eta(x)\leq1$, $\eta(x)=1$ for $x\in B_{\frac{1}{2}}$ and $\eta(x)=0$ for $x\in B_{\frac{3}{4}}^c$. Let $\rho_\epsilon(x)= \eta(\frac{x}{2\epsilon})$.
For a given test function $\varphi(x)\in \mathcal{D}(B_1)$, letting $\psi(x)=\varphi(x)-[\varphi(0)+\nabla\varphi(0)\cdot x]\eta(x)\in \mathcal{D}(B_1)$, then we have
\begin{equation}\label{0.11}
\begin{array}{cl}
&\displaystyle\int_{B_1}u(x)[(-\Delta)\varphi(x)-\vec b(x)\cdot\nabla\varphi(x)]dx\\
=&\displaystyle\int_{B_1}u(x)(-\Delta-\vec b(x)\cdot\nabla)(\varphi(0)\eta(x)+x\cdot\nabla\varphi(0) \eta(x))dx\\[3.5mm]
&\displaystyle+\int_{B_1}u(x)(-\Delta-\vec b(x)\cdot\nabla)\psi(x)dx\\
=&\displaystyle\varphi(0)\underbrace{\int_{B_1}u(x)(-\Delta-\vec b(x)\cdot\nabla)\eta(x)dx}_{{\rm denoted\ as\ }a_1 }+\nabla\varphi(0)\cdot\underbrace{\int_{B_1}u(x)(-\Delta-\vec b(x)\cdot\nabla)(x\eta(x))dx}_{{\rm denoted\ as\ }\vec d_1}\\
&\displaystyle+\lim\limits_{\epsilon\rightarrow0}\int_{B_1}u(x)(-\Delta-\vec b(x)\cdot\nabla)((1-\rho_\epsilon(x))\psi(x))dx\\
&\displaystyle+\underbrace{\lim\limits_{\epsilon\rightarrow0}\int_{B_1}u(x)(-\Delta-\vec b(x)\cdot\nabla)(\rho_\epsilon(x)\psi(x))dx}_{{\rm denoted\ as}\ I}\\
=&\displaystyle \varphi(0)a_1+\nabla\varphi(0)\cdot\vec d_1
+I+\int_{B_1}\psi(x)d\lambda\\

=&\displaystyle\varphi(0)a_1+\nabla\varphi(0)\cdot\vec d_1+I+\int_{B_1}\varphi(x)d\lambda-\varphi(0)\underbrace{\int_{B_1}\eta(x)d\lambda}_{{\rm denoted\ as\ }a_2}-\nabla\varphi(0)\cdot\underbrace{ \int_{B_1}x\eta(x)d\lambda}_{{\rm denoted\ as\ }\vec d_2}\\
=&\displaystyle\varphi(0)(a_1-a_2)-\nabla\varphi(0)\cdot(\vec d_2-\vec d_1)+I+\int_{B_1}\varphi(x)d\lambda.
\end{array}
\end{equation}
Here, we should notice that $a:=a_1-a_2$ and $\vec d:=\vec d_2-\vec d_1$ are finite and independent of $\varphi(x)$. Then, it suffices to show that $I=0$.

Indeed, one derives from $\psi(x)=\varphi(x)-\eta(x)[\varphi(0)+x\cdot\nabla\varphi(0)]$  that $|\psi(x)|\leq C|x|^2$, $|\nabla \psi(x)|\leq C|x|$ and $|\Delta \psi(x)|\leq C$. Combining with the fact  that $\rho_\epsilon(x)$ is supported in $B_\epsilon$, we have
\begin{equation}
|(-\Delta-\vec b(x)\cdot\nabla)(\rho_\epsilon(x)\psi(x))|\leq C_0,
\end{equation}
and hence
 \begin{equation}\label{0.12}
|I|\leq\lim\limits_{\epsilon\rightarrow0}\int_{B_\epsilon}u(x)|(-\Delta-\vec b(x)\cdot\nabla)(\rho_\epsilon(x)\psi(x))|dx\leq C_0\lim\limits_{\epsilon\rightarrow0}\int_{B_\epsilon}u(x)dx=0
 \end{equation}
Therefore, we have completed \textit{Step 1}.

\textit{Step 2.} We prove that the vector $\vec{d}$ in (\ref{0.9}) must be zero.

First, recall the fundamental solution $\Phi(x)$ of the equation $-\Delta u=0$, i.e.
\begin{equation}
\Phi(x)=
\frac{1}{n(n-2)\omega_n}|x|^{2-n}
\end{equation}
And in this step, we let $p$ be a real number satisfying $1<p<\frac{n}{n-1}$. Since $\vec{b}(x)\in C^1(B_1)$ and $\Phi(x)\in W^{1,p}(B_1)$, the $W^{2,p}$-theory enables us to choose $g(x)\in W^{2,p}(B_1)\cap W^{1,p}_0(B_1)$ such that
\begin{equation}
-\Delta g(x)+\vec b(x)\cdot\nabla g(x)=\vec{b}(x)\cdot\nabla\Phi(x),\ \ {\rm in}\ B_1.
\end{equation}

Next, let $h(x)=\Phi(x)-g(x)$, then $h(x)\in W^{1,p}(B_1)$ and satisfies
\begin{equation}
-\Delta h(x)+\vec b(x)\cdot\nabla h(x)=\delta_0(x),\ \ {\rm in}\ \mathcal{D}'(B_1).
\end{equation}

Denoting $w(x)=u(x)-ah(x)-\vec{d}\cdot\nabla h(x)$, then after direct calculations we derive that
\begin{equation}
\label{lck001}
-\Delta w(x)+\vec b(x)\cdot\nabla w(x)=\nu, \ \ {\rm in}\ \mathcal{D}'(B_1),
\end{equation}
where $\nu$, given by $d\nu=d\mu+[-c(x)u(x)+(d^i\partial_i\vec{b}(x))\cdot\nabla h(x)]dx$, is also a Radon measure.

From the regularity estimates for (\ref{lck001}) (see  Lemma \ref{lckest2}), we know $w(x)\in W^{1,p}(B_{\frac{1}{2}})$. Then from the above estimates for $w(x)$, $h(x)$ and $g(x)$, we have
\begin{equation}\label{lck002}
u(x)-\vec{d}\cdot\nabla\Phi(x)=w(x)+ah(x)-\vec d\cdot\nabla g(x)\in L^{\frac{n}{n-1}}(B_{\frac{1}{2}}).
\end{equation}

However, assuming $\vec d\neq0$  one sees that
\begin{equation}  \label {lck003}
\| (\vec d \cdot\nabla\Phi)^-\|_{L^{\frac{n}{n-1}} (B_{\frac{1}{2}})} = +\infty.
\end{equation}

Combining (\ref{lck002}) and (\ref{lck003}), we get
\begin{equation}
\|u^-(x)\|_{L^{\frac{n}{n-1}} (B_1) }
\geq\|(\vec d \cdot\nabla\Phi)^-\|_{L^{\frac{n}{n-1}} (B_{\frac{1}{2}})}-\|u(x)-\vec d \cdot\nabla\Phi(x)\|_{L^{\frac{n}{n-1}} (B_1)}=+\infty,
\end{equation}
which clearly contradicts  to $u(x)\geq 0$.
Therefore, we must have $\vec{d}=0$:
\begin{equation}\label{2.18}
-\Delta u(x)+\vec b(x)\cdot\nabla u(x)+c(x)u(x)=\mu+a\delta_0(x),\ \ \ {\rm in}\ \mathcal{D}'(B_1).
\end{equation}

\textit{Step 3.} To prove $a\geq0$ in (\ref{2.18}), we first notice that $\tilde\mu$ defined as
\begin{equation}\tilde\mu(E):=\mu(E)-\int_{E}c(x)u(x)dx+a\chi_{E}(0)\end{equation}
 is also a Radon measure, and  Lemma (\ref{lckest2}), (\ref{2.18}) yield $u(x)\in W^{1,p}(B_{\frac{1}{2}})$ for $1<p<\frac{n}{n-1}$.
These facts enable us to define $\lambda$ as:
\begin{equation}\lambda(E):=\mu(E\cap B_{\frac{1}{2}})-\int_{E\cap B_{\frac{1}{2}}}(c(x)u(x)+\vec b(x)\cdot\nabla u(x))dx.
\end{equation}
We know $\lambda$ is a Radon measure supported in $B_{\frac{1}{2}}$ satisfying $\lambda(\{0\})=0$ and $|\lambda|(B_{\frac{1}{2}})<+\infty$.

Defining $\displaystyle v(x)=\int_{\mathbb{R}^n}\Phi(x-y)d\lambda$, one calculates
\begin{equation}
-\Delta v(x)=\lambda,\ \ \ {\rm in}\ \mathcal{D}'(B_{\frac{1}{2}}).
\end{equation}
Hence, denoting $H(x)=u(x)-v(x)-a\Phi(x)$,
we see that
\begin{equation}
-\Delta H(x)=0,\ \ \ {\rm in}\ \mathcal{D}'(B_{\frac{1}{2}}).
\end{equation}
Consequently, $H(x)\in C^\infty(B_{\frac{1}{2}}) $.

Suppose that $a<0$, then we can derive a contradiction by computing the integral of $u(x)$ in a $\delta$-ball. Indeed, it holds that
\begin{equation}\label{2.19}
\begin{array}{rl}
\displaystyle\frac{1}{\delta^2}\int_{B_{\delta }} |v(x)| dx
\leq & \displaystyle\frac{1}{\delta^2} C_1\Big\{\int_{B_{ \delta } }  \Big[\int_{B_{ 2\epsilon } }\frac{1}{|x-y|^{n-2}}d|\lambda_y|\Big]dx+\int_{B_{ \delta } }  \Big[\int_{B_{ 2\epsilon }^c }  \underbrace{\frac{1}{|x-y|^{n-2}}}_{|x-y|\geq \epsilon}d|\lambda_y|\Big]dx\Big\} \\
\leq&\displaystyle\frac{C_1} { \delta^2}  \Big\{\int_{B_{ 2\epsilon } }  \Big[\int_{B_{ \delta } }\frac{1}{|x-y|^{n-2}}dx\Big]d|\lambda_y|  + C_2\frac{|\lambda|(B_{\frac{1}{2}})}{\epsilon^{n-2}}\delta^n\Big\}.
\end{array}
\end{equation}

Noticing that
\begin{equation}
\int_{B_\delta}\frac{1}{|x-y|^{n-2}}dx\leq\int_{B_{\delta}}\frac{1}{|x|^{n-2}}dx=C_3\delta^2,
\end{equation}
we get

\begin{equation}\label{2.19(0)}
\begin{array}{rl}
\displaystyle\frac{1}{\delta^2}\int_{B_{\delta }} |v(x)| dx
\leq&\displaystyle C_4(|\lambda|(B_{2\epsilon}) +\frac{\delta^{n-2}}{\epsilon^{n-2}}).
\end{array}
\end{equation}

However, considering $a\Phi(x)$ in the $ \delta$ ball, we have
\begin{equation}\label{2.20}
\frac{1}{\delta^2}\int_{B_{\delta}}  a\Phi(x)dx=\frac{C_1}{\delta^2}a\int_{B_\delta}\frac{1}{|x|^{n-2}}dx  \leq - C_5.
\end{equation}
By choosing $\epsilon$ sufficiently small first, then letting $\delta\ll\epsilon$ be small enough, we can get
\begin{equation}\label{2.21}
C_4(|\lambda|(B_{ 2\epsilon})+\frac{\delta^{n-2}}{\epsilon^{n-2}})< \frac{C_5}{2}.
\end{equation}
From (\ref{2.19})-(\ref{2.21}), it holds that
\begin{equation}\label{2.22}
\frac{1}{\delta^2}\int_{B_{\delta}} u(x)dx
\leq\displaystyle\frac{1}{\delta^2}\int_{B_{\delta}}(|v(x)| + a \Phi(x) + |H(x)| ) dx
\leq\displaystyle -\frac {C_5}{2}+\underbrace{\frac{1}{\delta^2}\int_{B_\delta}|H(x)|dx}_{\rightarrow 0,\ {\rm as}\ \delta\rightarrow 0}.
\end{equation}

Then, for sufficiently small $\delta$, we get$
\frac{1}{\delta^2}\int_{B_{\delta}} u(x)dx<0
$,
which contradicts to $u(x)\geq0$.

In summary, we obtain the desired result: for a constant $a\geq0$,
\begin{equation*}
-\Delta u(x)+\vec b(x)\cdot\nabla u(x)+c(x)u(x)=\mu+a\delta_0(x), \ \ {\rm in}\ \mathcal{D}'(B_1).
\end{equation*}
The proof of   $n=2$ is similar.
This completes the proof of Theorem \ref{l 0.0}.
\end{proof}

{\footnotesize REMARK} \textsc{1}. {\em Here $n\geq2$ is necessary, and we have the following counterexamples for $n=1$}:
\begin{itemize}
\item$u(x)=|x|$, and $-u''(x)=-2\delta_0(x)$, i.e. $a=-2<0$, in (\ref{0.10});\\
\item$u(x)=\left\{\begin{array}{cc}
0,\ x>0,\\
1,\ x<0,\end{array}
\right.$
and $-u''(x)=\delta_0'(x)$, i.e. $d=1\neq0$, in (\ref{0.10});\\
\item$u(x)=|x|^\theta,\ 0<\theta<1$, then $-u''(x)=\theta(1-\theta)|x|^{\theta-2}$. Let $\mu:=\theta(1-\theta)|x|^{\theta-2}dx$ then $\mu(B_{\frac{1}{2}})=\infty$, i.e. $\mu$ cannot be extended as a Radon measure on $B_1$.
\end{itemize}

 \subsection{Proof of Theorem \ref{l 0.1}}

\quad\!\!\quad In this section, we prove Theorem \ref{l 0.1}, which deals with the case of fractional Laplacian\cite{caffarelli1}. For $u\in C_0^\infty(\mathbb{R}^n)$, the fractional Laplacian $(-\Delta)^su$ is given by
\begin{equation}\label{1.0}
(-\Delta)^su(x)=C_{n,s}{\rm P.V.}\int_{\mathbb{R}^n}\frac{u(x)-u(y)}{|x-y|^{n+2s}}dy,\ \ \ 0<s<1,
\end{equation}
where P.V. stands for the Cauchy principle value. Let
\begin{equation*}
\mathcal{L}_\alpha=\bigg\{u: \mathbb{R}^n\rightarrow\mathbb{R}\bigg|\int_{\mathbb{R}^n}\frac{|u(y)|}{1+|y|^{n+\alpha}}dy<+\infty\bigg\}.
\end{equation*}
It is easy to see that for $w\in \mathcal{L}_{2s}$, $(-\Delta)^sw$ as a distribution is well-defined: $\forall \varphi\in C_0^\infty(\mathbb{R}^n)$,
\begin{equation}
(-\Delta)^sw(\varphi)=\int_{\mathbb{R}^n}w(x)(-\Delta)^s\varphi(x) dx.
\end{equation}

\noindent The proof of theorem \ref{l 0.1}:
\begin{proof}
The local integrability of $u(x)$ follows immediately from the fact that $u(x)\in\mathcal{L}_{2s}$. We will first show the existence of $\mu$. As in Theorem \ref{l 0.0}, applying the Riesz-Markov-Kakutani representation theorem, we see that
 \begin{equation}\label{1.4}
 (-\Delta)^s u(x)+\vec b(x)\cdot\nabla u(x) + c(x)u(x)=\mu\ \ {\rm in}\ \mathcal{D}'(B_1\backslash\{0\}),
 \end{equation}
where $\mu$ is a positive Radon measure defined on $B_1\backslash\{0\}$. Extend $\mu$ with $\mu(\{0\})=0$, we now show that $\mu(B_{r})<+\infty$ for some $0<r<1$.

Choose a constant $\alpha$ satisfying $0<\alpha<n-2s$, and define two functions \begin{equation}\varphi_\epsilon(x):=1-\big(\frac{\epsilon}{|x|}\big)^\alpha;\end{equation} \begin{equation}\psi_\epsilon(x)=\max\{\varphi_\epsilon(x),0\}.\end{equation}
After a direct computation, we have
\begin{equation}\label{1.3}
\begin{array}{rl}
(-\Delta)^s\varphi_\epsilon(x)-\vec b(x)\cdot\nabla\varphi_\epsilon(x)+2M\varphi_\epsilon(x)\leq 2M,
\end{array}
\end{equation}
when $|x|\leq r_0<1$ for some $r_0$ depending only on $s$, $\alpha$ and $M$.

 Applying Corollary \ref{lckcor4} to $\varphi_\epsilon(|x|)$, and denoting $\psi_\epsilon^\delta(x)= J_\delta\psi_\epsilon(x)$, one  derives, for $|x|\leq r_0<\frac{1}{2}$
\begin{equation}\label{1.5}
(-\Delta)^s\psi_\epsilon^\delta(x)-\vec b(x) \cdot\nabla\psi_\epsilon^\delta(x)+2M\psi_\epsilon^\delta(x)\leq 2M+M\delta \|\nabla\varphi_\epsilon(x)\|_{L^\infty(B_\epsilon^c)}\chi_{B_\epsilon}^\delta(x).\end{equation}

Then we define a cutoff function $\eta(x)\in C_0^\infty(B_1)$ satisfying $0\leq\eta(x)\leq1$, and
\begin{equation}
\eta(x)= \left \{
\begin{array} {ll}
1, \ \  |x| \leq \frac{1}{2}r_0,\\[2mm]
0, \ \ |x| \geq \frac{3}{4}r_0.
\end{array}  \right.
\end{equation}

Letting $0<\delta<\frac{\epsilon}{2}$, then $\eta(x)\psi_\epsilon^\delta(x)\in \mathcal{D}(B_1\backslash\{0\})$. Noticing also $\eta(x)\psi_\epsilon^\delta(x)\geq0$, we then test equation (\ref{1.4}) with it to have
\begin{equation}\label{1.6}
\begin{array}{rl}
&\displaystyle \int_{B_1}\psi_\epsilon^\delta(x)\eta(x)d\mu+\int_{B_1}\psi_\epsilon^\delta(x)\eta(x)[2M-c(x)-{\rm div}(\vec b(x))]u(x)dx\\
=&\displaystyle \int_{\mathbb{R}^n}u(x)[(-\Delta)^s(\psi_\epsilon^\delta(x)\eta(x))-\vec b(x)\cdot\nabla(\psi_\epsilon^\delta(x)\eta(x))+2M(\psi_\epsilon^\delta(x)\eta(x))]dx\\
=&\displaystyle \int_{\mathbb{R}^n}u(x)\eta(x)\underbrace{[(-\Delta)^s\psi_\epsilon^\delta(x)-\vec b(x)\cdot\nabla\psi_\epsilon^\delta(x)+2M\psi_\epsilon^\delta(x)]}_{\leq 2M+M\delta \|\nabla\varphi_\epsilon(x)\|_{L^\infty(B_\epsilon^c)}\chi_{B_\epsilon}^\delta(x),\ {\rm according\ to\ (\ref{1.5})}}dx\\
&\displaystyle+\int_{\mathbb{R}^n}u(x)\psi_\epsilon^\delta(x)(-\Delta)^s\eta(x)dx-\int_{\mathbb{R}^n}u(x)\psi_\epsilon^\delta(x)\underbrace{\vec b(x)\cdot\nabla\eta(x)}_{C_0^1(B_{r_0}\backslash B_{\frac{r_0}{4}})}dx\\
&\displaystyle-\int_{\mathbb{R}^n}\int_{\mathbb{R}^n}u(x)\frac{(\eta(x)-\eta(y))(\psi_\epsilon^\delta(x)-\psi_\epsilon^\delta(y))}{|x-y|^{n+2s}}dydx\\
\leq &\displaystyle \big(2M+ M\delta \|\nabla\varphi_\epsilon(x)\|_{L^\infty(B_\epsilon^c)}\big)\int_{B_1}u(x)dx+C_1\int_{B_{r_0}\backslash B_{\frac{r_0}{4}}}u(x)dx\\
&\displaystyle+\int_{\mathbb{R}^n}u(x)\bigg[\underbrace{|\psi_\epsilon^\delta(x)(-\Delta)^s\eta(x)|}_{{\rm denoted\ as\ }K_1}
+\underbrace{\int_{\mathbb{R}^n}\frac{|(\eta(x)-\eta(y))(\psi_\epsilon^\delta(x)-\psi_\epsilon^\delta(y))|}{|x-y|^{n+2s}}dy}_{{\rm denoted\ as\ }K_2}\bigg]dx.
\end{array}
\end{equation}
We now show that $K_i\leq\frac{C}{1+|x|^{n+2s}}$ $(i=1,2)$. In fact, $K_1\leq\frac{C}{1+|x|^{n+2s}}$ follows immediately from the fact that $\eta(x)\in C_0^\infty(\mathbb{R}^n)$ and that $0\leq\psi_\epsilon^\delta(x)\leq1$. For $K_2$, we can also derive $K_2\leq\frac{C}{1+|x|^{n+2s}}$ by considering the following three cases:

Case 1: $|x|\leq \frac{r_0}{4}$. Noticing that $\eta(z)=1$ for $|z|\leq\frac{1}{2}$, we have
\begin{equation}\label{1.6(0)}
\begin{array}{rl}
K_2=\!&\displaystyle\underbrace{\int_{|x-y|\leq\frac{r_0}{4}}\!\!\!\!\frac{|(\eta(x)\!-\!\eta(y))(\psi_\epsilon^\delta(x)\!-\!\psi_\epsilon^\delta(y))|}{|x-y|^{n+2s}}dy}_{=\ 0,\ {\rm since}\ |x|\leq\frac{r_0}{4}\ {\rm and}\ |y|\leq\frac{r_0}{2},\ {\rm then}\ \eta(x)=\eta(y)=1}
+ \int_{|x-y|>\frac{r_0}{4}}\!\!\!\!\frac{|(\eta(x)\!-\!\eta(y))(\psi_\epsilon^\delta(x)\!-\!\psi_\epsilon^\delta(y))|}{|x-y|^{n+2s}}dy\\
\leq\!&\displaystyle \int_{|x-y|>\frac{r_0}{4}}\frac{1}{|x-y|^{n+2s}}dy\leq C.
\end{array}
\end{equation}

Case 2: $\frac{r_0}{4}<|x|\leq 2$. We can estimate $K_2$ as
\begin{equation}\label{1.6(1)}
\begin{array}{rl}
K_2=\!&\displaystyle\int_{|x-y|\leq\frac{r_0}{8}}\!\!\!\!\frac{|(\eta(x)\!-\!\eta(y))(\psi_\epsilon^\delta(x)\!-\!\psi_\epsilon^\delta(y))|}{|x-y|^{n+2s}}dy
+\int_{|x-y|>\frac{r_0}{8}}\!\!\!\!\frac{|(\eta(x)\!-\!\eta(y))(\psi_\epsilon^\delta(x)\!-\!\psi_\epsilon^\delta(y))|}{|x-y|^{n+2s}}dy\\
\leq\! &\displaystyle \int_{|x-y|\leq\frac{r_0}{8}}\frac{\|\nabla\eta\|_{L^\infty(B_{r_0/8})}\|\nabla\psi_\epsilon^\delta\|_{L^\infty(B_{r_0/8})}}{|x-y|^{n+2s-2}}dy
+\int_{|x-y|>\frac{r_0}{8}}\frac{1}{|x-y|^{n+2s}}dy\\
 \leq\!&\displaystyle C.
 \end{array}
\end{equation}

Case 3: $|x|>2$. $K_2$ can be estimated as
\begin{equation}\label{1.6(2)}
\begin{array}{rl}
K_2=\!&\displaystyle \int_{|y|\leq1}\!\!\!\!\frac{|(\eta(x)\!-\!\eta(y))(\psi_\epsilon^\delta(x)\!-\!\psi_\epsilon^\delta(y))|}{|x-y|^{n+2s}}dy +\int_{|y|>1}\!\!\!\!\frac{|(\eta(x)\!-\!\eta(y))(\psi_\epsilon^\delta(x)\!-\!\psi_\epsilon^\delta(y))|}{|x-y|^{n+2s}}dy\\
=\! &\displaystyle \int_{|y|\leq1}\frac{\eta(y)|(\psi_\epsilon^\delta(x)\!-\!\psi_\epsilon^\delta(y))|}{|x-y|^{n+2s}}dy +\underbrace{\int_{|y|>1}\frac{\eta(y)|(\psi_\epsilon^\delta(x)\!-\!\psi_\epsilon^\delta(y))|}{|x-y|^{n+2s}}dy}_{=\ 0,\ {\rm since}\ \eta(y)=0\ {\rm when}\ |y|>1}\\
 \leq\!&\displaystyle C|x|^{-(n+2s)}\int_{|y|\leq1}dy\leq C\frac{1}{1+|x|^{n+2s}}.
 \end{array}
\end{equation}
Thus:
\begin{equation}\label{1.6(3)}
K_i\leq C\frac{1}{1+|x|^{n+2s}},\ \ i=1,2.
\end{equation}

Combining (\ref{1.6}) and (\ref{1.6(3)}), and noticing that $u(x)\in \mathcal{L}_{2s}$, we obtain
\begin{equation}\label{1.7}
\begin{array}{rl}
&\displaystyle \int_{B_1}\psi_\epsilon^\delta(x)\eta(x)d\mu+\int_{B_1}\psi_\epsilon^\delta(x)\eta(x)[2M-c(x)-{\rm div}(\vec b(x))]u(x)dx\\
&\qquad\leq C+C\delta \|\nabla\varphi_\epsilon(x)\|_{L^\infty(B_\epsilon^c)}.
\end{array}
\end{equation}
Taking $\delta\rightarrow0$ first and then $\epsilon\rightarrow0$ in (\ref{1.7}), we can derive $\mu({B_{\frac{r_0}{2}}})<+\infty$ for some constant $r_0>0$, where we have used the facts that $\psi_\epsilon^\delta\rightarrow\psi_\epsilon$ in $C(B_1)$ as $\delta\rightarrow0$ and $\psi_\epsilon\rightarrow 1$ monotonically as $\epsilon\rightarrow0$.

In what follows, similar to the proof of theorem \ref{l 0.0}, we prove (\ref{eq111}) in three steps.

\textit{Step 1}:  We claim that there exists constants $a$ and $\vec{d}$, independent of $\varphi(x)$, such that
\begin{equation}\label{1.8}
(-\Delta)^su(x)+\vec{b}(x) \nabla u(x)+c(x)u(x)=\mu_0+a\delta_0+\vec{d}\cdot\nabla\delta_0, \ \ {\rm in}\ \mathcal{D}'(B_1).
\end{equation}

We first rewrite  (\ref{1.4}) as
\begin{equation}\label{1.10}
(-\Delta)^s u(x)+{\rm div}(\vec b(x)u(x))=[-c(x)+{\rm div}(\vec b(x))]u(x)+\mu:=\lambda\ \ {\rm in}\ \mathcal{D}'(B_1\backslash \{0\}),
\end{equation}
where $\lambda$ is a signed Radon measure on $B_1$.
Define $\eta(x)\in  \mathcal{D}(B_1)$ to be a function  satisfying $0\leq\eta(x)\leq1$, $\eta(x)=1$ for $x\in B_{\frac{1}{2}}$ and $\eta(x)=0$ for $x\in B_{\frac{1}{4}}^c$. Then let $\rho_\epsilon(x)= \eta(\frac{x}{2\epsilon})$.
For a given test function $\varphi(x)\in \mathcal{D}(B_1)$, letting $\psi(x)=\varphi(x)-(\varphi(0)+\nabla\varphi(0)\cdot x)\eta(x)\in \mathcal{D}(B_1)$, then we have
\begin{equation}\label{1.11}
\begin{array}{rl}
&\displaystyle\int_{\mathbb{R}^n}u(x)[(-\Delta)^s\varphi(x)-\vec{b}(x)\varphi(x)]dx\\
=&\displaystyle\int_{\mathbb{R}^n}u(x)[(-\Delta)^s-\vec{b}(x) \cdot\nabla][(\varphi(0)+\nabla\varphi(0)\cdot x)\eta(x)+\psi(x)]dx\\
=&\displaystyle\varphi(0)\underbrace{\int_{\mathbb{R}^n}u(x)[(-\Delta)^s-\vec{b}(x)\cdot \nabla]\eta(x)dx}_{{\rm denoted\ as}\ a_1}
+\nabla\varphi(0)\cdot\underbrace{\int_{\mathbb{R}^n}u(x)[(-\Delta)^s-\vec{b}(x)\cdot \nabla](x\eta(x))dx}_{{\rm denoted\ as\ }\vec d_1}\\
&\displaystyle+\int_{\mathbb{R}^n}u(x)[(-\Delta)^s-\vec{b}(x)\cdot \nabla]\psi(x)\\
=&\displaystyle\varphi(0)a_1+\nabla\varphi(0)\cdot\vec d_1
+\underbrace{\lim\limits_{\epsilon\rightarrow0}\int_{\mathbb{R}^n}u(x)[(-\Delta)^s-\vec{b}(x)\cdot \nabla](\rho_\epsilon(x) \psi(x))dx}_{{\rm denoted\ as\ }I}\\
&\displaystyle+\lim\limits_{\epsilon\rightarrow0}\int_{\mathbb{R}^n}u(x)[(-\Delta)^s-\vec{b}(x)\cdot \nabla][(1-\rho_\epsilon(x))\psi(x)]dx\\
=&\displaystyle\varphi(0)a_1+\nabla\varphi(0)\cdot\vec d_1+I+\int_{\mathbb{R}^n}\psi(x)d\lambda\\
=&\displaystyle\varphi(0)a_1+\nabla\varphi(0)\cdot\vec d_1+I+\int_{\mathbb{R}^n}\varphi(x)d\lambda-\varphi(0)\underbrace{\int_{\mathbb{R}^n}\eta(x)d\lambda}_{{\rm denoted\ as\ }a_2}
-\nabla\varphi(0)\cdot\underbrace{\int_{\mathbb{R}^n}x\eta(x)d\lambda}_{{\rm denoted\ as\ }\vec d_2}\\
=&\displaystyle\varphi(0)(a_1-a_2)-\nabla\varphi(0)\cdot(\vec{d_2}-\vec d_1)+I+\int_{\mathbb{R}^n}\varphi(x)d\lambda.
\end{array}
\end{equation}
Here, we should notice that $a:=a_1-a_2$ and $\vec d:=\vec d_2-\vec d_1$ are finite and independent of $\varphi(x)$. Then, it suffices to show that
$I=0$.
We now write $I$ as
\begin{equation}
I=\lim\limits_{\epsilon\rightarrow0}\int_{\mathbb{R}^n}u(x)(-\Delta)^s(\rho_\epsilon(x) \psi(x))dx-\lim\limits_{\epsilon\rightarrow0}\int_{B_1}u(x)\vec{b}(x)\cdot \nabla(\rho_\epsilon(x) \psi(x))dx=:I_1+I_2.
\end{equation}

With an estimation similar to those in the proof of (\ref{0.12}), we derives $I_2=0$. Thus, next we focus on $I_1$.

One derives from $\psi(x)=\varphi(x)-\eta(x)(\varphi(0)+x\cdot\nabla\varphi(0)$  that $|\psi(x)|\leq C|x|^2$, $|\nabla \psi(x)|\leq C|x|$ and $|\partial_{ij} \psi(x)|\leq C$. Combining with the fact  that $\rho_\epsilon(x)$ is supported in $B_\epsilon$, we have
\begin{equation}
\begin{array}{rl}
|\nabla( \rho_\epsilon(x)\psi(x))| \leq&\displaystyle |\nabla\rho_\epsilon(x)\psi(x)|+|\rho_\epsilon(x)\nabla \psi(x)|\\
\leq& C\big[\frac{1}{\epsilon}|\nabla\rho(\frac{x}{\epsilon})||x|^2+\rho(\frac{x}{\epsilon})|x|\big]\leq C\epsilon,
\end{array}
\end{equation}
\begin{equation}
\begin{array}{rl}
&\displaystyle|\partial_{ij}( \rho_\epsilon(x)\psi(x))| \\
\leq&\displaystyle |\partial_{i}\rho_\epsilon(x)\partial_{j}\psi(x)|+|\partial_{j}\rho_\epsilon(x)\partial_{i}\psi(x)|+|\partial_{ij}\rho_\epsilon(x) \psi(x)|+|\partial_{ij}\psi(x)\rho_\epsilon(x)|\\
\leq& C\big[\frac{1}{\epsilon}|\partial_i\rho(\frac{x}{\epsilon})||x|+\frac{1}{\epsilon}|\partial_j\rho(\frac{x}{\epsilon})||x|+\frac{1}{\epsilon^2}\partial_{ij}\rho(\frac{x}{\epsilon})|x|^2
+\rho(\frac{x}{\epsilon})\big]\leq C.
\end{array}
\end{equation}

We will show that
\begin{equation}\label{1.21}
|(-\Delta)^{s} ( \rho_{\epsilon}(x)\psi(x))|\leq  \frac{C(\epsilon^{2s}+\epsilon^{2-2s})}{1+|x|^{n+2s}}\end{equation}
by considering the following two cases.

Case 1: $|x|\leq2$.
By the definition of the fractional Laplacian, for $x\in B_{2\epsilon}$, we have
\begin{equation}\label{3.33}
\begin{array} { rl}
&(-\Delta)^{s} ( \rho_{\epsilon}(x)\psi(x))\\[1mm]
=&\displaystyle C_{n,s}\lim\limits_{\delta\rightarrow0}\int_{\mathbb{R}^n\backslash B_\delta}  \frac{\rho_{\epsilon} (x)\psi(x)-\rho_{\epsilon}(x+y)\psi(x+y)}{ |y|^{n+2s} }dy\\
=&\displaystyle \frac{C_{n,s}}{2}\underbrace{\lim\limits_{\delta\rightarrow0}\int_{B_{4\epsilon}\backslash B_\delta}  \frac{2\rho_{\epsilon} (x)\psi(x)-\rho_{\epsilon}(x+y)\psi(x+y)-\rho_{\epsilon}(x-y)\psi(x-y)}{ |y|^{n+2s} }dy}_{{\rm denoted\ as\ }K_1}\\
&+\ C_{n,s}\displaystyle\underbrace{\int_{B_{2\epsilon}^{c}} \frac{\rho_{\epsilon} (x)\psi(x)-\rho_{\epsilon}(x+y)\psi(x+y)}{ |y|^{n+2s} }dy}_{{\rm denoted\ as\ }K_2}
\end{array}
\end{equation}

By using the Taylor expansion, one can observe
\begin{equation}
\begin{array}{ rl}
&|2\rho_{\epsilon}(x)\psi(x)-\rho_{\epsilon}(x+y) \psi(x+y)-\rho_{\epsilon}(x-y) \psi(x-y)|\\[1mm]
\leq&C|\nabla^2( \rho_\epsilon(\xi)\psi(\xi))| |y|^2\leq C|y|^2.
\end{array}
\end{equation}

Hence
\begin{equation} \label{1.12}
|K_1|\leq C\int_{B_{2\epsilon}}  \frac {|y|^{2}} { |y| ^{n+2s}}dy \leq  C \epsilon^{2-2s}.
\end{equation}

For $K_2$, we have
\begin{equation}
|\rho_{\epsilon}(x)\psi(x)-\rho_{\epsilon}(x+y) \psi(x+y)|
\leq C|\nabla( \rho_\epsilon(\xi)\psi(\xi))| |y|\leq C\epsilon|y|,
\end{equation}
and hence, for $R>4$
\begin{equation}
\begin{array}{rl}
|K_2|\!\!\!\!&\displaystyle=\bigg|\int_{B_R^c}\frac{\rho_{\epsilon} (x)\psi(x)}{ |y|^{n+2s} }dy+\int_{B_R\backslash B_\epsilon} \frac{\rho_{\epsilon} (x)\psi(x)-\rho_{\epsilon}(x+y)\psi(x+y)}{ |y|^{n+2s} }dy\bigg|\\
&\displaystyle\leq \int_{B_R^c}\frac{|\rho_\epsilon(x)\psi(x)|}{ |y|^{n+2s} }dy+\int_{B_R\backslash B_\epsilon} \frac{|\rho_{\epsilon} (x)\psi(x)-\rho_{\epsilon}(x+y)\psi(x+y)|}{ |y|^{n+2s} }dy\\
&\displaystyle\leq \int_{B_R^c}\frac{1}{ |y|^{n+2s} }dy+C\epsilon\int_{B_R\backslash B_\epsilon}\frac{1}{|y|^{n+2s-1}}dy\\
&\displaystyle\leq  C_1\frac{1}{R^{2s}}+C_2\epsilon\Big|\frac{1}{R^{2s-1}}-\frac{1}{\epsilon^{2s-1}}\Big|\\
&\displaystyle\leq C\Big(\frac{1}{R^{2s}}+\frac{\epsilon}{R^{2s-1}}+\epsilon^{2-2s} \Big).
\end{array}
\end{equation}
Let $R=\epsilon^{-1}$, we get
\begin{equation}
 \label{1.13}
 |K_2|\leq C(\epsilon^{2s}+\epsilon^{2-2s}).
 \end{equation}

Combining (\ref{1.12}) and (\ref{1.13}), we obtain
\begin{equation} \label{1.14}
| (-\Delta)^{s} ( \rho_\epsilon(x)\psi(x))| \leq C(\epsilon^{2s}+\epsilon^{2-2s}).
\end{equation}

Case 2: $|x|>2$. Since $\psi(x)$ is compactly supported in $B_1$
\begin{equation}
\begin{array}{rl} \label{1.15}
|(-\Delta)^{s} ( \rho_{\epsilon}(x)\psi(x))|
=&C_{n,s}\displaystyle\bigg|\int_{B_1}\frac{\rho_\epsilon(y)\psi(y)}{|x-y|^{n+2s}}dy\bigg|
\leq C_{n,s}\displaystyle\int_{B_1}\frac{|\rho_\epsilon(y)\psi(y)|}{|x-y|^{n+2s}}dy\\
\leq&\displaystyle \frac{C}{1+|x|^{n+2s}}\int_{B_1}|\rho_\epsilon(y)\psi(y)|dy\leq \frac{C\epsilon}{1+|x|^{n+2s}},
\end{array}
\end{equation}
where we used the fact that $|x-y|\approx(|x|+1)$, for $|x|>2$ and $|y|\leq1$.

This completes the proof of (\ref{1.21}),
 and hence
\begin{equation}
I_1=\lim\limits_{\epsilon\rightarrow0}\int_{\mathbb{R}^n}u(x)(-\Delta)^s(\rho_\epsilon(x) \psi(x))dx =0.
\end{equation}
Thus \textit{Step 1} is completed.

\textit{Step 2.} We prove $\vec{d}=0$ (see \ref{1.8}). Let \begin{equation}\Phi(x)=\frac{C_{n,s}}{|x|^{n-2s}}\end{equation} be the fundamental solution of the fractional Laplacian $(-\Delta)^s$, and $p$ be a real number satisfying $1<p<\frac{n}{n-2s+1}$. Since $\vec b(x)\in C^1(B_1)$ and $\Phi(x)\in W^{1,p}(B_1)$, the basic potential theory  enables us to choose  $g\in\mathcal{L}_{2s}$ satisfying:
\begin{equation}\label{1.16}
(-\Delta)^s g(x)+\vec b(x)\cdot\nabla g(x)=\vec{b}(x)\cdot\nabla\Phi(x)\chi_{B_1}(x)\ \ {\rm in}\ \ \mathcal{D}'(\mathbb{R}^n),
\end{equation}
with $(-\Delta)^sg\in L^p(\mathbb{R}^n)$ (here $s>\frac{1}{2}$) and $g\in W^{1,p}(B_1)$.
Further, let $h=\Phi-g$, then $h\in W^{1,p}(B_1)$ and
\begin{equation}\label{1.17}
(-\Delta)^s h(x)+\vec b(x)\cdot\nabla h(x)=\delta_0(x),\ \ {\rm in}\ \ \mathcal{D}'(B_1).
\end{equation}

Letting $w(x)=u(x)-ah(x)-\vec d\cdot\nabla h(x)$, by a direct calculation, we derive

\begin{equation}\label{1.17(0)}
(-\Delta)^s w(x)+\vec b(x)\cdot\nabla w(x)=\mu-c(x)u(x)+(d^i\partial_i\vec b(x)),\ \ {\rm in}\ \ \mathcal{D}'(B_1).
\end{equation}

By using the estimate we derived  in Lemma \ref{l 8.2} for the equation (\ref{1.17(0)}),  we know $(-\Delta)^sw(x)\in L^p_{\rm loc}(B_1)$. From the above regularity estimates for $h(x)$ and $g(x)$, we have
\begin{equation}\label{1.17(1)}
u(x)-\vec{d}\cdot\nabla\Phi(x)=w(x)+ah(x)-\vec d\cdot\nabla g(x)\in L^{\frac{n}{n-2s+1}}(B_{\frac{1}{2}}).
\end{equation}

Assume $\vec d\neq0$ now, then one sees that
\begin{equation}  \label {1.17(2)}
\| (\vec d \cdot\nabla\Phi)^-\|_{L^{\frac{n}{n-2s+1}} (B_{\frac{1}{2}})} = +\infty.
\end{equation}

Thus, by combining (\ref{1.17(1)}) and (\ref{1.17(2)}), we obtain
\begin{equation}
\|u^-\|_{L^{\frac{n}{n-2s+1}} (B_1) }
\geq\|(\vec d \cdot\nabla\Phi)^-\|_{L^{\frac{n}{n-2s+1}} (B_{\frac{1}{2}})}-\|u-\vec d \cdot\nabla\Phi\|_{L^{\frac{n}{n-2s+1}} (B_1)}=+\infty,
\end{equation}
which clearly contradicts to $u(x)\geq0$.
Consequently, $\vec{d}=0$, i.e.,
\begin{equation}\label{1.22}
(-\Delta)^s u(x)+\vec b(x)\cdot\nabla u(x)+c(x)u(x)=\mu+a\delta_0(x),\ \ \ {\rm in}\ \mathcal{D}'(B_1).
\end{equation}

\textit{Step 3.} To prove $a\geq0$ in (\ref{1.22}), we first notice that $\tilde\mu$ defined as \begin{equation}\tilde\mu(E):=\mu(E)-\int_{E}c(x)u(x)dx+a\chi_{E}(0)\end{equation}
is also a Radon measure, and Lemma \ref{l 8.1} together with (\ref{1.22}) gives $u(x)\in W^{1,p}(B_{\frac{1}{2}})$ for $1<p<\frac{n}{n-2s+1}$.
These facts enable us to define $\lambda$ as:
\begin{equation}\lambda(E):=\mu(E\cap B_{\frac{1}{2}})-\int_{E\cap B_{\frac{1}{2}}}(c(x)u(x)+\vec b(x)\cdot\nabla u(x))dx.
\end{equation}
We know $\lambda$ is a Radon measure supported in $B_{\frac{1}{2}}$ satisfying $\lambda(\{0\})=0$ and $|\lambda|(B_{\frac{1}{2}})<+\infty$.

Defining $\displaystyle v(x)=\int_{\mathbb{R}^n}\Phi(x-y)d\lambda_y$, then one easily calculates
\begin{equation}
(-\Delta)^s v(x)=\mu-(c(x)u(x)+\vec b(x)\cdot\nabla u(x)),\ \ \ {\rm in}\ \mathcal{D}'(B_{\frac{1}{2}}).
\end{equation}
Then, denoting $H(x)=u(x)-v(x)-a\Phi(x)$, we know
\begin{equation}
(-\Delta)^s H(x)=0,\ \ \ {\rm in}\ \mathcal{D}'(B_{\frac{1}{2}}),
\end{equation}
and hence $H(x)\in C^\infty(B_{\frac{1}{2}})$.

Suppose that $a<0$, then we will derive a contradiction by computing the integral of $u(x)$ in a $\delta$-ball. In particular, we have
\begin{equation}\label{1.23}
\begin{array}{rl}
\displaystyle\frac{1}{\delta^{2s}}\int_{B_{\delta }} |v(x)| dx
\leq & \displaystyle\frac{1}{\delta^{2s}}  \Big\{\int_{B_{ \delta } }  \Big[\int_{B_{ 2\epsilon } }\frac{1}{|x-y|^{n-2s}}d|\lambda_y|\Big]dx+\int_{B_{ \delta } }  \Big[\int_{B_{ 2\epsilon }^c }  \underbrace{\frac{1}{|x-y|^{n-2s}}}_{|x-y|\geq \epsilon}d|\lambda_y|\Big]dx\Big\} \\
\leq&\displaystyle\frac{C_1} { \delta^{2s}}  \Big\{\int_{B_{ 2\epsilon } }  \Big[\int_{B_{ \delta } }\frac{1}{|x-y|^{n-2s}}dx\Big]d|\lambda_y|  + C_2\frac{|\lambda|(B_{\frac{1}{2}})}{\epsilon^{n-2s}}\delta^n\Big\}.
\end{array}
\end{equation}

Noticing that
\begin{equation}
\int_{B_\delta}\frac{1}{|x-y|^{n-2s}}dx\leq\int_{B_{2\delta}}\frac{1}{|x|^{n-2s}}dx=C_3\delta^{2s},
\end{equation}
we get

\begin{equation}\label{1.23(0)}
\begin{array}{rl}
\displaystyle\frac{1}{\delta^{2s}}\int_{B_{\delta }} |v(x)| dx
\leq&\displaystyle C_4(|\lambda|(B_{2\epsilon}) +\frac{\delta^{n-2s}}{\epsilon^{n-2s}}).
\end{array}
\end{equation}

However, considering $a\Phi(x)$ in the $ \delta$ ball, we have
\begin{equation}\label{1.24}
\frac{1}{\delta^{2s}}\int_{B_{\delta}}  a\Phi(x)dx=\frac{C_1}{\delta^{2s}}a\int_{B_\delta}\frac{1}{|x|^{n-2s}}dx  \leq - C_5.
\end{equation}
By choosing $\epsilon$ sufficiently small first and then letting $\delta$ be small enough, we can get
\begin{equation}\label{1.25}
C_4(|\lambda|(B_{ 2\epsilon})+\frac{\delta^{n-2s}}{\epsilon^{n-2s}})< \frac{C_5}{2}.
\end{equation}
Then from (\ref{1.23})-(\ref{1.25}), we have
\begin{equation}\label{1.26}
\frac{1}{\delta^{2s}}\int_{B_{\delta}} u(x)dx
\leq\displaystyle\frac{1}{\delta^{2s}}\int_{B_{\delta}}(|v(x)| + a \Phi(x) + |H(x)| ) dx
\leq\displaystyle -\frac {C_5}{2}+\underbrace{\frac{1}{\delta^{2s}}\int_{B_\delta}|H(x)|dx}_{\rightarrow 0,\ {\rm as}\ \delta\rightarrow 0}.
\end{equation}

Then, for sufficiently small $\delta$, we get \begin{equation}\label{1.27}
\frac{1}{\delta^{2s}}\int_{B_{\delta}} u(x)dx<0,
\end{equation}
which contradicts to $u(x)\geq 0$.

Consequently,   $a\geq0$ and:
\begin{equation*}
(-\Delta)^s u(x)+\vec b(x)\cdot\nabla u(x)+c(x)u(x)=\mu+a\delta_0(x),\ \ {\rm in}\ \ \mathcal{D}'(B_1).
\end{equation*}

This completes the proof of Theorem \ref{l 0.1}.
\end{proof}
\subsection{Two more adaptable B\^{o}cher type theorems}

The following two more adaptable  B\^{o}cher type theorems are direct corollaries of Theorem (\ref{l 0.0}) and Theorem (\ref{l 0.1})
 \begin{theorem}
Let $u(x),f(x)\in L_{{\rm loc}}^1(B_1\backslash \{0\})$ be a nonnegative function in $\mathbb{R}^n$ $(n\geq2)$ satisfying
\begin{equation}
-\Delta u(x)+\vec b(x)\cdot\nabla u(x) + c(x)u(x)=f(x)\geq0\ \ {\rm in}\ \mathcal{D}'(B_1\backslash \{0\}),
\end{equation}
where $\|\vec b(x)\|_{C^1(B_1)}+\|c(x)\|_{L^\infty(B_1)}\leq M$ for some constant $M$, then $u(x), f(x)\in L_{{\rm loc}}^1(B_1)$ and

\begin{equation}
-\Delta u(x)+\vec b(x)\cdot\nabla u(x) + c(x)u(x)=f(x)+a\delta_0(x)\ \ {\rm in}\ \mathcal{D}'(B_1),
\end{equation}
for some $a\geq0$.
\end{theorem}

 \begin{theorem}
Let $u(x)\in \mathcal{L}_{2s}$, $f(x)\in L_{\rm loc}^1(B_1\backslash\{0\})$ with $s\in(\frac{1}{2},1)$ be a nonnegative function in $\mathbb{R}^n$ ($n\geq2$) satisfying

\begin{equation}
(-\Delta)^s u(x)+\vec b(x)\cdot\nabla u(x) + c(x)u(x)=f(x)\geq0\ \ {\rm in}\ \ \mathcal{D}'(B_1\backslash \{0\}),
\end{equation}
where $\|\vec b(x)\|_{C^1(B_1)}+\|c(x)\|_{L^\infty(B_1)}\leq M$ for some constant $M$,  then $u(x), f(x)\in L_{{\rm loc}}^1(B_1)$ and

\begin{equation}
(-\Delta)^s u(x)+\vec b(x)\cdot\nabla u(x) + c(x)u(x)=f(x)+a\delta_0(x)\ \ {\rm in}\ \mathcal{D}'(B_1),
\end{equation}
for some $a\geq0$.

Besides, when $\vec{b}(x)\equiv0$, then the theorem holds for $s\in(0,1)$.
\end{theorem}

\section{Maximum principles}\label{3}

\quad\!\!\quad In this section, we  prove Theorem \ref{l 0.2} Theorem \ref{l 1.2} and Theorem \ref{l 0.2}. These  maximum principles can be derived easily if we apply the B\^{o}cher type theorems.

\subsection{Proof of Theorem \ref{l 0.2}}
\begin{proof}

Without loss of generality, we may assume $c(x)\equiv M$, $r=1$ and $x_0=0$.

Applying Theorem \ref{l 0.0}, we know $ u(x)\in L_{{\rm loc}}^1(B_1)$, and
\begin{equation}
-\Delta u(x)+\vec{b}(x)\cdot\nabla u(x)+Mu(x)=\mu+a\delta_0\geq0,\ \ {\rm in}\ \ \mathcal{D}' (B_1).
\end{equation}
Mollify this equation, we have
\begin{equation}
-\Delta J_\delta u(x)+J_\delta (\vec b\cdot\nabla u)(x)+MJ_\delta u(x)\geq0\ \ {\rm in}\ B_{1-\delta}.
\end{equation}
For $0<\delta<\frac{1}{4}$, denoting $u^\delta(x)=J_\delta u(x)$ and $N_\delta(x)=\vec b(x)\cdot \nabla J_\delta u(x)-J_\delta (\vec b\cdot \nabla u)(x)$, we know:
\begin{equation}
\left\{\ \begin{aligned}
&\!-\Delta u^\delta(x)+\vec{b}(x)\cdot\nabla u^\delta(x)+Mu^\delta(x)\geq N_\delta(x),\ \ &{\rm in}\ \ &B_{1-\delta}\\
&u(x)\geq m,\ \ &{\rm on}\ \ &\partial B_{1-\delta}.
\end{aligned}
\right.
\end{equation}

Let $v^\delta(x)$ be a solution of the following problem:
\begin{equation}
\left\{\ \begin{aligned}
&\!-\Delta v^\delta(x)+\vec{b}(x)\cdot\nabla v^\delta(x)+Mv^\delta(x)=N_\delta(x),\ \ &{\rm in}\ \ &B_{1-\delta}\\
&v(x)=0,\ \ &{\rm on}\ \ &\partial B_{1-\delta}.
\end{aligned}\right.
\end{equation}
Applying the standard maximum principle to $u^\delta-v^\delta$ on $B_{1-\delta}$, we know $u^\delta-v^\delta\geq Cm$ in $B_{1-\delta}$, where $C=C(n,M)$ is a positive constant. To derive (\ref{6.3}) we should notice that $C(n,0)=1$.

According to Lemma \ref{lcklem3}, $N_\delta(x)\rightarrow0$ in $L^1_{\rm loc}(B_1)$ as $\delta\rightarrow 0$, hence $v^\delta(x)\rightarrow 0$ in $W^{1,p}_{\rm loc}(B_r)$ as $\delta\rightarrow0$ for given $0<r<1$ and $1<p<\frac{n}{n-1}$. Now for any given $\varphi(x)\in\mathcal{D}(B_1)$ with $\varphi(x)\geq0$, say ${\rm supp}( \varphi ) \subset B_r$ for some $r<1$,  we have
\begin{equation}\label{6.4}
\int_{B_r}(u^\delta(x)-v^\delta(x))\varphi(x)dx\geq Cm\int_{B_r}\varphi(x)dx.
\end{equation}
Here, it's clear that $\int_{B_r}u^\delta(x)\varphi(x)dx\rightarrow \int_{B_r}u(x)\varphi(x)dx$ and $\int_{B_r}v^\delta(x)\varphi(x)dx\rightarrow 0$ as $\delta\rightarrow0$. Hence, letting $\delta\rightarrow0$ in (\ref{6.4}), we get
\begin{equation}
\int_{B_1}u(x)\varphi(x)dx\geq Cm\int_{B_1}\varphi(x)dx,
\end{equation}
which immediately implies (\ref{6.2}).
\end{proof}

{\footnotesize REMARK} \textsc{5}. {\em Theorem \ref{l 0.2} does not hold when $n=1$. In fact, for the special case $c(x)\equiv0$, the function $v(x)=|x|$ satisfies (\ref{6.1}) but obviously does not satisfy (\ref{6.2})}.

\subsection{Proof of Theorem \ref{l 1.3}}

\begin{proof}Define:
\begin{gather*}
M_1:=\int_{B_{\frac{7}{8}}^c}\frac{C_{n,s}}{(|y|-\frac{5}{8})^{n+2s}}dy<\infty,\qquad M_2:=\int_{B_{\frac{7}{8}}\backslash B_{\frac{5}{8}} }\frac{C_{n,s}}{|y|^{n+2s}}dy>0.
\end{gather*}
Then there exists a constant: $\sigma\in(0,1)$ such that:  $\displaystyle\int_{B_1\backslash B_\sigma}|\nabla j(y)|dy\leq\frac{M_2}{2M}.$\\
Further, for such fixed $\sigma$, we define: $\displaystyle M_3:=\sup_{y\in B_\sigma}\frac{|\nabla j(y)|}{j(y)}<\infty.$

We prove that $u(x)\geq \alpha m$ with
$\displaystyle\alpha=\frac{M_2}{2(M_1+M_2+2M+MM_3)}$.

Without loss of generality, we assume $r=1$ and $x_0=0$.

Let $v(x)=\min\{u(x),m\}$, from Lemma \ref{lckcor2} one derives:
\begin{equation}
(-\Delta)^s v(x)+\vec{b}(x)\cdot\nabla v(x)+Mv(x)\geq0,\ \ {\rm in}\ \ \mathcal{D}' (B_1).
\end{equation}
Consider the mollified function $v^\delta(x)=J_\delta v(x)$  with $0<\delta\leq\frac{1}{8}$, we have:
\begin{equation}
\left\{\ \begin{aligned}
&(-\Delta)^sv^\delta(x)+J_\delta(\vec{b}\cdot\nabla v)(x)+Mv^\delta(x)\geq 0,\ &{\rm in}\ \  &B_{\frac{7}{8}},\\
&v^\delta(x)=m>0,\ &{\rm in}\ \ &B_{\frac{7}{8}}\backslash B_{\frac{5}{8}},\\
\end{aligned}
\right.
\end{equation}

We need to prove that:
\begin{equation}
v^\delta(\bar{x})=\min\big\{v^\delta(x)\big|x\in\overline{B_{\frac{7}{8}}}\big\}\geq\alpha m.\end{equation}
Because $v^\delta(x)=m\geq v^\delta(y)$ for all $x\in B_{\frac{7}{8}}\backslash B_{\frac{5}{8}}$ and $y\in B_{\frac{7}{8}}$, we only need to consider the case $\bar{x}\in B_{\frac{5}{8}}$.
Then $\bar{x}$ is an interior minimal point, thus $\nabla v^\delta(\bar{x})=0$. We calculate:
\begin{equation}
\begin{aligned}
J_\delta(\vec b\cdot\nabla v)(\bar x)&=J_\delta(\vec b\cdot\nabla v)(\bar x)-\vec b(\bar x)\cdot\nabla v^\delta(\bar x)\\
&=\int_{\mathbb{R}^n}v(y)[(\vec b(\bar x)-\vec b(y))\cdot\nabla_yj_\delta(\bar x-y)-(\nabla\cdot\vec b)(y)j_\delta(\bar x-y)]\ dy\\
&\leq \sum_{i,j}\|\partial_ib_j\|_{L^{\infty}(B_1)}\int_{\mathbb{R}^n}v(y)[\delta|\nabla j_\delta(\bar x-y)|+j_\delta(\bar x-y)]\ dy\\
&\leq M\bigg[v^\delta(\bar x)+\underbrace{\int_{B_{\delta}(\bar x)}v(y)\delta|\nabla j_\delta(\bar x-y)|dy}_{{\rm denote\ as}\ I}\bigg].
\end{aligned}
\end{equation}

We can then estimate $I$:
\begin{equation}\label{lckmp2.2}\begin{aligned}
I&=\int_{B_{\delta}(\bar x)}v(y)\delta|\nabla j_\delta(\bar x-y)|dy\\
&\leq\int_{B_{\sigma\delta}(\bar x)}v(y)\delta|\nabla j_\delta(\bar x-y)|dy+m\int_{B_{\delta}(\bar x)\backslash B_{\sigma\delta}(\bar x)}\delta|\nabla j_\delta(\bar x-y)|dy\\
&\leq M_3\int_{B_{\sigma\delta}(\bar x)}v(y) j_\delta(\bar x-y) dy+\frac{M_2m}{2M}\leq M_3v^\delta(\bar x)+\frac{M_2m}{2M}.
\end{aligned}\end{equation}

Further,  define
\begin{gather*}
M_4(\bar x)=\int_{B_{\frac{7}{8}}^c}\frac{C_{n,s}}{|\bar x-y|^{n+2s}}dy\leq M_1,   \qquad  M_5(\bar x)=\int_{B_{\frac{7}{8}}\backslash B_{\frac{5}{8}} }\frac{C_{n,s}}{|\bar x-y|^{n+2s}}dy\geq M_2,
\end{gather*}

Then
\[
\begin{array}{rl}
0 \leq& (-\Delta)^sv^\delta(\bar{x})+J_\delta(\vec b\cdot\nabla v)(\bar x)+Mv^\delta(\bar{x})\\
= &\displaystyle C_{n,s}\bigg(\int_{B_{\frac{7}{8}}^c}\frac{v^\delta(\bar{x})-v^\delta(y)}{|\bar{x}-y|^{n+2s}}dy
+\int_{B_{\frac{7}{8}}\backslash{B_{\frac{5}{8}}}}\frac{v^\delta(\bar{x})-v^\delta(y)}{|\bar{x}-y|^{n+2s}}dy
\\&\displaystyle+{\rm P.V.}\int_{B_{\frac{5}{8}}}\underbrace{\frac{v^\delta(\bar{x})-v^\delta(y)}{|\bar{x}-y|^{n+2s}}}_{\leq0}dy\bigg)+M[v^\delta(\bar x)+I]+Mv^\delta(\bar{x})\\
\leq &M_4(\bar{x})v^\delta(\bar{x})+M_5(\bar{x})(v^\delta(\bar{x})-m)+0+2Mv^\delta(\bar{x})+MI\\
\leq &\displaystyle M_1v^\delta(\bar{x})+M_2(v^\delta(\bar{x})-m)+2Mv^\delta(\bar{x})+MM_3v^\delta(\bar x)+\frac{M_2m}{2}\\
=&\displaystyle (M_1+M_2+2M+MM_3)v^\delta(\bar x)-\frac{M_2}{2}m.
\end{array}
\]
This implies $v^\delta(x)\geq\alpha m$ in $B_{\frac{7}{8}}$.
By taking $\delta\rightarrow0$, we get $v(x)\geq \alpha m$, thus $u(x)\geq \alpha m$ in $B_1$.
\end{proof}
\subsection{Proof of Theorem \ref{l 1.2}}
\begin{proof}
Applying Theorem \ref{l 0.1}, we know:
\begin{equation}
(-\Delta)^s u(x)+\vec{b}(x)\cdot\nabla u(x)+Mu(x)=\mu+a\delta_0\geq0,\ \ {\rm in}\ \ \mathcal{D}' (B_1).
\end{equation}
Then we derive the desired result from theorem \ref{l 1.3}
\end{proof}

We emphasize here the importance of the B\^{o}cher type Theorem \ref{l 0.1}:
{\em the nonnegative fractional super-harmonic function on the punctured ball $B_1\backslash\{0\}$ is also actually a fractional super-harmonic function on the whole ball $B_1$.}

 {\footnotesize REMARK} \textsc{6}. {\em There is an obvious difference between (\ref{6.3}) in Theorem \ref{l 0.2} and (\ref{6.8}) in Theorem \ref{l 1.2} for the special case $M=0$ in (\ref{6.1}) and (\ref{6.7}), that is, the positive constant $\alpha<1$ in (\ref{6.8}), which is resulted by the non-locality of the fractional Laplacian.  See some interesting examples in \cite{li0}.}

\section{Regularity of solutions}\label{4}
 \begin{lemma}\label{lckest2}
 Suppose that $u(x)\in L^1_{\rm loc}(B_1)$ is a solution of the equation
 \begin{equation}
 -\Delta u(x)+\vec{b}(x)\cdot\nabla u(x)=\mu, \ \ \ {\rm in}\  \ \mathcal{D}'(B_1),
 \end{equation}
 where $\mu$ is a Radon measure and $\vec{b}(x)\in C^1(B_1)$. Then
 it holds that $u(x)\in W^{1,p}_{\rm loc}(B_1)$ for  $1\leq p<\frac{n}{n-1}$.
 \end{lemma}

   \begin{lemma}\label{a18.0} Suppose that $w(x)\in \mathcal{L}_{2s}$ with $0<s<1$ and
   \begin{equation}\label{8.3}
   (-\Delta)^sw(x)=0,\ {\rm in}\ B_1,
   \end{equation}
then $w(x)$ is smooth in $B_1$.
 \end{lemma}
\begin{proof}[Proof of Lemma \ref{a18.0}] First of all, we consider the case that $w(x)\in C(\mathbb{R}^n)$. Then we can use Green's function of the fractional Laplacian given in Theorem 2.10 of \cite{bucur} to have
\begin{equation}\label{8.4}
w(x)=\int_{B_t^c}P_t(x,y)w(y)dy,
\end{equation}
with the Green's function $P_t(x,y)$ with fixed $t>0$ given by
\begin{equation}\label{8.5}
P_t(x,y)=c(n,s)\bigg(\frac{t^2-|x|^2}{|y|^2-t^2}\bigg)^s\frac{1}{|x-y|^n},\ {\rm for\ any}\ x\in B_t\ {\rm and\ any}\ y\in\mathbb{R}^n\backslash\bar{B}_t.
\end{equation}

For given $x\in B_1$, by taking average of $w(x)$ on the torus $1-\delta_0\leq t\leq 1$ with $\delta_0=\frac{1-|x|}{2}$ from the representation (\ref{8.4}), we have
\begin{equation}\label{8.6}
\begin{array}{rl}
w(x)=&\displaystyle\int_{1-\delta_0}^1dt\frac{1}{\delta_0}\int_{B_t^c}P_t(x,y)w(y)dy\\
=&\displaystyle\int_{|y|>1-\delta_0}w(y)\frac{1}{\delta_0}\int_{1-\delta_0\leq t\leq\min(|y|,1)}P_t(x,y)dtdy
:=\displaystyle\int_{|y|\geq 1-\delta_0}w(y)G_{1,\delta_0}(x,y)dy,
\end{array}
\end{equation}
with
\begin{equation}\label{8.7}
G_{r,\delta_0}(x,y)=\frac{1}{\delta_0}\int_{r-\delta_0\leq t\leq\min(|y|,r)}P_t(x,y)dt.
\end{equation}
By a direct calculation, we know that $\partial^\alpha_x G_{r,\delta_0}(x,y)$ is continuous in $y$ for any nonnegative multi-index $\alpha$. This yields that $w(x)$ is smooth in $B_{t}$ with $0<t<1$.

Second, when $w(x)\in\mathcal{L}_{2s}$, we consider the mollification $w_\epsilon(x)=J_\epsilon w(x)$ with $\epsilon\ll\delta_0$. Then
\begin{equation*}
w_\epsilon(x)=\int_{|y|\geq 1-2\delta_0}w_\epsilon(y)G_{1-\delta_0,\delta_0}(x,y)dy\rightarrow \int_{|y|\geq 1-2\delta_0}w(y)G_{1-\delta_0,\delta_0}(x,y)dy,\ {\rm when}\ \epsilon\rightarrow0.
\end{equation*}
Here we have used the fact that $w_\epsilon(x)\rightarrow w(x)$ in $\mathcal{L}_{2s}$, which is derived from $G_{1-\delta_0,\delta_0}(x,y)\leq \frac{C}{1+|y|^{n+2s}}$. This proves that $w(x)$ is smooth
in $B_1$. \end{proof}

\begin{lemma}\label{l 8.1}
Let $B_1\subset \mathbb{R}^n$ with $n\geq 2$.
Assume that $u(x)\in \mathcal{L}_{2s}$ with $s\in (\frac{1}{2},1)$ satisfies
\begin{equation}
(-\Delta)^su(x)+\vec{b}(x)\cdot\nabla u(x)=\mu\quad \text{in}\ \mathcal{D}'(B_1),
\end{equation}
 where $\mu$ is a Radon measure and $\vec{b}(x)\in C^1(B_1)$.

 Then $(-\Delta)^ru(x)\in L^p_{\rm loc}(B_1)$ for $1\leq p<\frac{n}{n-2s+2r}$ and $r<s$.
 \end{lemma}
\begin{proof}
First, we show that $(-\Delta)^ru(x)\in L^1_{\rm loc}(B_1)$ for $r<s$. To this end, we fix $\epsilon\in(0,\frac{1}{3})$ and set
\begin{equation}\label{8.10}
\begin{array}{rl}
v(x)=&\displaystyle c(n,s)\int_{\mathbb{R}^n}\frac{1}{|x-y|^{n-2s}}d\tilde{\mu}_y+c(n,s)\int_{\mathbb{R}^n}{\rm div}_y\bigg(\frac{\vec{b}(y)}{|x-y|^{n-2s}}\bigg)\tilde{u}(y)dy\\
=&\displaystyle c(n,s)\int_{B_1}\frac{1}{|x-y|^{n-2s}}d\tilde{\mu}_y+c(n,s)\int_{B_1}{\rm div}_y\bigg(\frac{\vec{b}(y)}{|x-y|^{n-2s}}\bigg)\tilde{u}(y)dy\\
:=&\displaystyle v_1(x)+v_2(x)
\end{array}
\end{equation}
with
\begin{equation*}
\tilde{u}(x)=\begin{cases}
u(x),\ &x\in B_1,\\
0,\  &x\in B_1^c,
\end{cases}
\qquad {\rm and}\qquad
\tilde{\mu}(E)=\mu(E\cap B_{1-\epsilon})
\end{equation*}
By Young's inequality, $v_1(x)$ is well-defined. A direct calculation shows
\begin{equation*}(-\Delta)^sv(x)+\vec{b}(x)\cdot\nabla u(x)=\mu\quad \text{in}\ \mathcal{D}'(B_{1-\epsilon}),
\end{equation*}
\begin{equation}\label{8.11}
\text{and\ consequently}\qquad(-\Delta)^sw(x)=0\quad \text{in}\ \mathcal{D}'(B_{1-\epsilon}),\quad {\rm where}\quad w(x):=u(x)-v(x).
\end{equation}
From Lemma \ref{a18.0}, we know that $w(x)$ is smooth in $B_{1-\epsilon}$. We now consider the estimate of $v(x)$.

For $v_1(x)$, we have
\begin{equation}\label{8.12}
(-\Delta)^rv_1(x)=C\int_{B_{\frac{3}{4}}}\frac{1}{|x-y|^{n-2s+2r}}d\mu_y,
\end{equation}
which together with Hausdorff-Young inequality gives $(-\Delta)^rv_1(x)\in L^1(B_1)$ for $0<r<s$.

For $v_2(x)$, it needs to use method of induction. We rewrite $v_2(x)$ as
\begin{equation}\label{8.13}
\begin{array}{rl}
v_2(x)=&\displaystyle\int_{B_1}\frac{{\rm div}\vec{b}(y)}{|x-y|^{n-2s}}u(y)dy+C\int_{B_1}\frac{\vec{b}(y)\cdot(x-y)}{|x-y|^{n-2s+2}}u(y)dy\\
:=&\displaystyle v_3(x)+v_4(x).
\end{array}
\end{equation}
The estimate for $v_3(x)$ can be obtained in the same way of $v_1(x)$. We are focused to show $(-\Delta)^rv_4\in L^1(B_1)$ In fact, by using Hausdorff-Young inequality and the following estimate
\begin{equation}\label{8.14}
|(-\Delta)^{\frac{2s-1}{4}}v_4(x)|\leq C\int_{B_1}\frac{|\vec{b}(y)|}{|x-y|^{n-s+\frac{1}{2}}}|u(y)|dy.
\end{equation}
This implies $(-\Delta)^{\frac{2s-1}{4}}v_4(x)\in L^p(B_1)$ for $1\leq p<\frac{n}{n-(s-\frac{1}{2})}$. Similarly, we have for some integer $k>0$ such that
\begin{equation}\label{8.15}
(-\Delta)^{\frac{k(2s-1)}{4}}v_4(x)\in L^p(B_1),\ {\rm with}\ 1\leq p<\frac{n}{n-k(s-\frac{1}{2})}\ {\rm and}\ \frac{k(2s-1)}{2}<1.
\end{equation}

Then we set $\frac{c(2s-1)}{4}=r$, and denote $[c]=k$. Hence, by virtue of the following equality
\begin{equation}\label{8.16}
\begin{array}{rl}
(-\Delta)^{r}v_4(x)=&\displaystyle\int_{B_1}(-\Delta)_x^{\frac{(c-k)(2s-1)}{4}}\frac{\vec{b}(y)\cdot(x-y)}{|x-y|^{n-2s+2}}(-\Delta)_y^{\frac{k(2s-1)}{4}}u(y)dy,
\end{array}
\end{equation}
(\ref{8.15}) and Hausdorff-Young inequality, we get
\begin{equation}\label{8.17}
(-\Delta)^{r}v_4(x)\in L^p(B_1),\ {\rm for}\ 1\leq p<\frac{n}{n-2s+2r}\ {\rm with}\ 0<r<s,
\end{equation}
which together with the above estimates for $v_1,v_3$ implies $(-\Delta)^rv(x)\in L^1(B_1)$ for $0<r<s$. Since $u(x)=v(x)+w(x)$ and $w(x)$ is smooth in $B_t$ with $0<t<1$, we can conclude that
\begin{equation}\label{8.18}
(-\Delta)^ru(x)\in L^1(B_t)\ {\rm for}\ 0<t<1,\ {\rm and}\ 0<r<s.
\end{equation}

In the second step, we will prove that $(-\Delta)^ru(x)\in L^p(B_t)$ for $0<t<1$ and $0<r<s$. In fact, we can select $r'$ such that $0<r<r'<s$ and $(-\Delta)^{r'}u(x)\in L^1(B_t)$. Then select $1\leq p<\frac{n}{n-2(r'-r)}$. By virtue of Hardy-Littlewood-Sobolev inequalities, we have that $(-\Delta)^ru(x)\in L^p(B_t)$ for $1\leq p<\frac{n}{n-2s+2r}$ and $0<r<s$. This completes the proof of Lemma \ref{l 8.1}.
\end{proof}

In the same way, we can also obtain $W^{\alpha,p}$-estimate when the source term is a function in $L^p(B_1)$-space.

  \begin{lemma}\label{l 8.2}
Let $B_1\subset \mathbb{R}^n$ with $n\geq 2$.
Assume that $u(x)\in \mathcal{L}_{2s}$ with $s\in (\frac{1}{2},1)$ satisfies
\begin{equation}
(-\Delta)^su(x)+\vec{b}(x)\cdot\nabla u(x)=f(x)\quad \text{in}\ \mathcal{D}'(B_1),
\end{equation}
 where $f(x)\in L^p(B_1)$ with $p>1$ and $\vec{b}(x)\in C^1(B_1)$.

 Then $(-\Delta)^su(x)\in L^p_{\rm loc}(B_1)$ for $p>1$.
 \end{lemma}

\begin{proof} As in Lemma \ref{l 8.1}, we first set
\begin{equation}\label{8.21}
\begin{array}{rl}
v(x)=&\displaystyle\int_{\mathbb{R}^n}\frac{f(y)}{|x-y|^{n-2s}}dy+\int_{\mathbb{R}^n}{\rm div}_y\bigg(\frac{\vec{b}(y)}{|x-y|^{n-2s}}\bigg)\tilde{u}(y)dy\\
=&\displaystyle\int_{B_1}\frac{f(y)}{|x-y|^{n-2s}}dy+\int_{B_1}{\rm div}_y\bigg(\frac{\vec{b}(y)}{|x-y|^{n-2s}}\bigg)u(y)dy\\
:=&\displaystyle v_1(x)+v_2(x)
\end{array}
\end{equation}
with
\begin{equation*}
\tilde{u}(x)=
\begin{cases}
u(x),\ &x\in B_1,\\
0,\ &x\in B_1^c.
\end{cases}
\end{equation*}
It is obvious that $(-\Delta)^sv(x)+\vec{b}(x)\cdot\nabla u(x)=f(x)$ and
\begin{equation}\label{8.22}
(-\Delta)^sw(x)=0\ \ {\rm in}\ B_1,\ \ {\rm where}\ \ w(x)=u(x)-v(x).
\end{equation}
From Lemma \ref{a18.0}, we know that $w(x)$ is smooth in $B_t$ with $0<t<1$. Thus, in the following, we consider the estimate of $v(x)$.

For $v_1(x)$, since $f(x)\in L^p(B_1)$, we know that $(-\Delta)^sv_1(x)\in L^p(B_1)$ by using Hausdorff-Young inequality again.

For $v_2(x)$, by taking method of induction to the following equality as in Lemma \ref{l 8.1}
\begin{equation}\label{8.23}
\begin{array}{rl}
v_2(x)=&\displaystyle\int_{B_1}\frac{{\rm div}\vec{b}(y)}{|x-y|^{n-2s}}u(y)dy+C\int_{B_1}\frac{\vec{b}(y)\cdot(x-y)}{|x-y|^{n-2s+2}}u(y)dy\\
:=&\displaystyle v_3(x)+v_4(x),
\end{array}
\end{equation}
we can conclude that there exists a constant $r$ satisfying $s<r<\frac{n}{2}$ such that $(-\Delta)^rv_4(x)\in L^q(B_1)$ for $1<q<\frac{n}{n-2r}$. Then by using Hardy-Littlewood-Sobolev inequality $\|(-\Delta)^sv_4\|_{L^p(B_1)}\leq C\|(-\Delta)^rv_4\|_{L^q(B_1)}$ when we further select $r$ satisfying $\max\{\frac{n}{4}(1-\frac{1}{p})+\frac{s}{2},\ s\}<r<\frac{n}{2}$. Thus, $(-\Delta)^sv_4(x)\in L^p(B_1)$. Similarly, we can also get $(-\Delta)^sv_3(x)\in L^p(B_1)$. Thus from $v(x)=v_1(x)+v_3(x)+v_4(x)$, we have proved that $(-\Delta)^sv(x)\in L^p(B_1)$.

Since $u(x)=v(x)+w(x)$ and $w(x)$ is smooth in $B_t$ with $0<t<1$, we can conclude that
\begin{equation}\label{8.24}
(-\Delta)^su(x)\in L^p(B_1),\ 0<t<1.
\end{equation}
This proves Lemma \ref{l 8.2}.
\end{proof}

\section{Constructions of super/sub harmonic functions}\label{5}

\quad\!\!\quad Given two superharmonic functions $u$ and $v$,
it is well known that one can construct a new superharmonic function $w(x)=\min\{u(x),v(x)\}$.
 Here, we present a few interesting lemmas to give a somewhat complete presentation for both Laplacian and fractional Laplacian cases
 via the more general distributional approach.
\subsection{Statements of the lemma}

\begin{lemma}\label{lcklem1}
Let $\Omega$ be a domain in $\mathbb{R}^n$. Assume $u(x), v(x) , f(x) ,g(x) \in C^2_{0}(\Omega) $, and satisfy
\begin{equation}\label{lcklem1.1}
\begin{array}{rl}
{\rm i)}&\displaystyle -\Delta u(x)+\vec{b}(x)\cdot\nabla u(x)+c(x)u(x)=f(x)\ \ {\rm in}\ \ \Omega,\\
{\rm ii)}&\displaystyle -\Delta v(x)+\vec{b}(x)\cdot\nabla v(x)+c(x)v(x)=g(x)\ \ {\rm in}\ \ \Omega,
\end{array}
\end{equation}
where $\|\vec b(x)\|_{C^1(\Omega)}+\|c(x)\|_{L^\infty(\Omega)}<\infty$.  Then for $w(x)=\max\{u(x),v(x)\}$, it holds that
\begin{equation} \label{lcklem1.2}
-\Delta w(x)+\vec{b}(x)\cdot\nabla w(x)+c(x)w(x)\leq f(x)\chi_{u>v} + g(x)\chi_{u\leq v}\ \ {\rm in}\ \mathcal{D}'(\Omega).
\end{equation}
\end{lemma}

\begin{lemma}\label{lcklem2}
Let $\Omega$ be a domain in $\mathbb{R}^n$. Assume $u(x), v(x) , f(x) ,g(x) \in C^2_{0}(\Omega) $, and satisfy
\begin{equation}\label{lcklem2.1}
\begin{array}{rl}
{\rm i)}&\displaystyle (-\Delta)^su(x)+\vec{b}(x)\cdot\nabla u(x)+c(x)u(x)=f(x)\ \ {\rm in}\ \ \Omega,\\
{\rm ii)}&\displaystyle (-\Delta)^sv(x)+\vec{b}(x)\cdot\nabla v(x)+c(x)v(x)=g(x)\ \ {\rm in}\ \ \Omega,
\end{array}
\end{equation}
where $\|\vec b(x)\|_{C^1(\Omega)}+\|c(x)\|_{L^\infty(\Omega)}<\infty$.  Then for $w(x)=\max\{u(x),v(x)\}$, it holds that
\begin{equation} \label{lcklem2.2}
(-\Delta)^sw(x)+\vec{b}(x)\cdot\nabla w(x)+c(x)w(x)\leq f(x)\chi_{u>v} + g(x)\chi_{u\leq v}\ \ {\rm in}\ \mathcal{D}'(\Omega).
\end{equation}
\end{lemma}

\begin{theorem}\label{lckcor1}
Let $\Omega$ be a domain in $\mathbb{R}^n$. Assume $u(x), v(x) , f(x) ,g(x) \in L^1_{\rm loc}(\Omega) $, and satisfy
\begin{equation}\label{lckcor1.1}
\begin{array}{rl}
&\displaystyle -\Delta u(x)+\vec{b}(x)\cdot\nabla u(x)+c(x)u(x)\leq f(x)\ \ {\rm in}\ \mathcal{D}'(\Omega),\\
&\displaystyle -\Delta v(x)+\vec{b}(x)\cdot\nabla v(x)+c(x)v(x)\leq g(x)\ \ {\rm in}\ \mathcal{D}'(\Omega),
\end{array}
\end{equation}
where $\|\vec b(x)\|_{C^1(\Omega)}+\|c(x)\|_{L^\infty(\Omega)}<\infty$.  Then for $w(x)=\max\{u(x),v(x)\}$, it holds that
\begin{equation} \label{l1}
-\Delta w(x)+\vec{b}(x)\cdot\nabla w(x)+c(x)w(x)\leq f(x)\chi_{u> v} + g(x)\chi_{u<v}+\max\{f(x),g(x)\}\chi_{u=v}\ \ {\rm in}\ \mathcal{D}'(\Omega).
\end{equation}
\end{theorem}

\begin{theorem}\label{lckcor2}
Let $\Omega$ be a domain in $\mathbb{R}^n$. Assume $u,v\in\mathcal{L}_{2s}$, $f, g\in L^1_{\rm loc}(\Omega)$, and satisfy
\begin{equation}\label{lckcor2.1}
\begin{array}{rl}
&\displaystyle (-\Delta)^su(x)+\vec{b}(x)\cdot\nabla u(x)+c(x)u(x)\leq f(x)\ \ {\rm in}\ \mathcal{D}'(\Omega),\\
&\displaystyle (-\Delta)^sv(x)+\vec{b}(x)\cdot\nabla v(x)+c(x)v(x)\leq g(x)\ \ {\rm in}\ \mathcal{D}'(\Omega),
\end{array}
\end{equation}
where $\|\vec b(x)\|_{C^1(\Omega)}+\|c(x)\|_{L^\infty(\Omega)}<\infty$.  Then for $w(x)=\max\{u(x),v(x)\}$, it holds that
\begin{equation}
(-\Delta)^sw(x)+\vec{b}(x)\cdot\nabla w(x)+c(x)w(x)\leq f(x)\chi_{u> v} + g(x)\chi_{u<v}+\max\{f(x),g(x)\}\chi_{u=v}\ {\rm in}\ \mathcal{D}'(\Omega).
\end{equation}
\end{theorem}
\subsection{Some important estimates appear in the proof}


\begin{lemma}\label{l 3.1}
Let $\Omega$ be a domain in $\mathbb{R}^n$. Assume $w(x)\in \mathcal{L}_{2s}$, $\ f(x)\in L_{{\rm loc}}^1(\Omega)$, and satisfy
\begin{equation}
(-\Delta)^sw(x)=f(x)\ \ \ {\rm in}\ \mathcal{D}'(\Omega)
\end{equation}
then
\begin{equation}
(-\Delta)^sJ_\delta w(x)=J_\delta f(x)\ \ {\rm in}\ \ \Omega^\delta,
\end{equation}
where $\Omega^\delta=\{x\in\mathbb{R}^n|B_\delta(x) \subset \Omega\}$.

In particular, if $(-\Delta)^sw(x)\geq0$, i.e. $w(x)$ is fractional super-harmonic in the domain $\Omega\subset\mathbb{R}^n$, then $J_\delta w(x)$ is also fractional super-harmonic in the domain $\Omega^\delta$.
\end{lemma}
\begin{proof}
By a direct computation, we know
\begin{equation}\label{I 3.1(lck)}
\begin{array}{rl}
&\displaystyle\int_{\mathbb{R}^n\backslash \{|x-y|\leq\epsilon\}}\frac{1}{|x-y|^{n+2s}}(J_\delta w(x)-J_\delta w(y)) dy\\
=&\displaystyle\int_{\mathbb{R}^n\backslash \{|x-y|\leq\epsilon\}}\frac{1}{|x-y|^{n+2s}}\Big(\int_{\mathbb{R}^n} j^\delta(x-z)w(z) dz -\int_{\mathbb{R}^n} j^\delta(y-z)w(z)dz\Big) dy\\
=&\displaystyle \int_{\mathbb{R}^n}  w(z)\int_{\mathbb{R}^n\backslash \{|x-y|\leq\epsilon\}}\frac{j^\delta(x-z)-j^\delta(y-z)}{|x-y|^{n+2s}}dy dz.
\end{array}
\end{equation}

Taking $\epsilon\rightarrow 0$ in (\ref{I 3.1(lck)}) and noticing the definition of  the fractional Laplacian, we know
\begin{equation}
 (-\Delta)^sJ_\delta w(x)=\int_{\mathbb{R}^n}w(z)  (-\Delta)_x^sj^\delta(x-z)  dz=\int_{\mathbb{R}^n}w(z)  (-\Delta)_z^sj^\delta(x-z)  dz,
\end{equation}
where we have used the fact that $(-\Delta)_x^sg(x-z)=(-\Delta)_z^sg(x-z)$. Notice also:
\begin{equation}
J_\delta f(x)=
J_\delta((-\Delta)^sw)(x)=\int_{\mathbb{R}^n}w(z)(-\Delta)_z^sj^\delta(x-z)dz\ \ {\rm in}\ \ \Omega^\delta.
\end{equation}

Lemma \ref{l 3.1} follows.
\end{proof}

\begin{lemma}\label{lcklem3}
Let $\Omega$ be a domain in $\mathbb{R}^n$,  $u(x)\in L^1_{\rm loc}(\Omega)$ and $\|\vec b(x)\|_{C^1(\Omega)}<\infty$, then
\begin{equation}
 N_\delta(x):=\vec b(x)\cdot\nabla J_\delta u(x)-J_\delta(\vec b\cdot\nabla u)(x)\rightarrow 0\ \ {\rm in}\ L^1_{\rm loc}(\Omega),\ {\rm as}\ \delta\rightarrow0^+.
\end{equation}
\end{lemma}
\begin{proof}
For ant  subdomain $\Omega'\subset\subset\Omega$. Define a real number $d=\frac{1}{2}d(\partial\Omega'',\bar{\Omega'})>0$, and let $\delta\in(0,d)$. For any $x\in\Omega'$, we  calculate
\begin{equation}
\begin{array}{rl}
N_\delta(x)=&\displaystyle\vec b(x)\cdot\nabla J_\delta u(x)-J_\delta (\vec b\cdot\nabla u)(x)\\
=&\displaystyle\int_{B_\delta(x)} u(y)\vec b(x)\cdot\nabla_xj^\delta(x-y) dy+\int_{B_\delta(x)} u(y){\rm div}_y(\vec b(y)j^\delta(x-y)) dy\\
=&\displaystyle-\int_{B_\delta(x)} u(y)\vec b(x)\cdot\nabla_yj^\delta(x-y) dy+\int_{B_\delta(x)} u(y){\rm div}_y(\vec b(y)j^\delta(x-y)) dy\\
=&\displaystyle\int_{B_\delta(x)} u(y){\rm div}_y[(\vec b(y)-\vec b(x))j^\delta(x-y)]dy\\
=&\displaystyle\int_{B_\delta(x)}(u(y)-u(x)){\rm div}_y[(\vec b(y)-\vec b(x))j^\delta(x-y)]dy.
\end{array}
\end{equation}
Hence
\begin{equation}
\begin{array}{rl}
|N_\delta(x)|\leq&\displaystyle\int_{B_\delta(x)}|u(y)-u(x)|\underbrace{\big|{\rm div}_y[(\vec b(y)-\vec b(x))j^\delta(x-y)]\big|}_{\leq C_1\delta^{-n}}dy\\
\leq&\displaystyle\frac{ C_1 }{\delta^n}\int_{B_\delta(x)}|u(y)-u(x)|dy=:M_\delta(x).
\end{array}
\end{equation}

Lebesgue differentiation theorem shows that $M_\delta(x)\rightarrow0$ as $\delta\rightarrow0^+$  at Lebesgue points of $u(x)$. Choose a subsequence $\delta_k\rightarrow0$ such that $\lim_{k\rightarrow\infty}\int_{\Omega'}M_{\delta_k}(x)dx=\limsup_{\delta\rightarrow 0}\int_{\Omega'}M_\delta(x)dx$. The Egorov's theorem then gives that $\forall \epsilon>0, \exists A_\epsilon\subset\Omega'\ (\mu(A_\epsilon)<\epsilon)$ and $M_{\delta_k}(x)$ converges to $0$ uniformly on $\Omega'\backslash A_\epsilon$ as $\delta\rightarrow0^+$, and

\begin{equation}\label{lcklem3.1}
\begin{array}{rl}
\displaystyle\int_{\Omega'} M_{\delta_k}(x)dx
\leq&\displaystyle\int_{\Omega'\backslash A_\epsilon}M_{\delta_k}(x) dx+\frac{C_1}{\delta_k^n}\int_{A_\epsilon}\int_{B_{\delta_k}}|u(x+\xi)|+|u(x)|d\xi dx\\
\leq&\displaystyle\int_{\Omega'\backslash A_\epsilon}M_{\delta_k}(x) dx+2C_2\sup_{\xi\in B_{\delta_k}}\int_{A_\epsilon}|u(x+\xi)|dx.
\end{array}
\end{equation}

Letting  $k\rightarrow\infty$  in (\ref{lcklem3.1}), we get
\begin{equation}\label{lcklem3.2}
\limsup_{\delta\rightarrow0^+}\int_{\Omega'}M_\delta(x)dx
\leq 2C_2\sup_{\xi\in B_d}\int_{A_\epsilon}|u(x+\xi)|dx.
\end{equation}

Taking $\epsilon\rightarrow0^+$ in (\ref{lcklem3.2}) then, we derive by the integrability $u\in L^1_{\rm loc}(\Omega)$ that,
\begin{equation}
\lim_{\epsilon\rightarrow0}\sup_{\xi\in B_d}\int_{A_\epsilon}|u(x+\xi)|dx=0,
\end{equation}
hence
\begin{equation}
\limsup_{\delta\rightarrow0^+}\int_{\Omega'}|N_\delta(x)|dx\leq\limsup_{\delta\rightarrow0^+}\int_{\Omega'}M_\delta(x)dx
\leq 0.
\end{equation}

This proves Lemma \ref{lcklem3}.
\end{proof}

\begin{lemma}\label{lcklem4}
Let $\Omega$ be a domain in $\mathbb{R}^n$, $u(x)$ smooth and $\|\vec b(x)\|_{C^1(\Omega)}\leq K$, then for $x\in\Omega$,
\begin{equation}\label{lcklem4.1}
|J_\delta(\vec b\cdot\nabla u^+)(x)-\vec b(x)\cdot\nabla J_\delta u^+(x)|\leq  K\delta\|\nabla u\|_{L^\infty(\{ u>0\})}\chi_{ u>0 }^\delta(x).
\end{equation}
\end{lemma}
\begin{proof}The proof is by direct calculation.
\end{proof}

\subsection{Proof of the lemmas}
\subsubsection{The proof of Lemma \ref{lcklem1}}
\begin{proof}
We now prove the lemma in three steps.

\textit{Step 1.} We first assume that $v(x)=g(x)=0$, $\partial\{x\in\Omega|u(x)>0\}$ is an $n-1$ dimensional $C^1$-manifold and derive the desired results (\ref{lcklem1.2}).
Let $\varphi(x)\in C_0^{\infty}(\Omega)$ be a nonnegative test function and $\nu(x)$ denote the exterior unit normal vector of $\partial\{x\in\Omega|u(x)>0\}$.
Then we have
\begin{equation}\label{lcklem1.3}
\begin{array}{rl}
&\!\!\!\!\displaystyle\int_{\Omega}-w(x)\Delta\varphi(x)dx
=\int_{\{x\in\Omega|u(x)>0\}}-u(x)\Delta\varphi(x)dx\\
&\displaystyle =\int_{\{x\in\Omega|u(x)>0\}}-\varphi(x)\Delta u(x) dx+\int_{\partial\{x\in\Omega|u(x)>0\}}\varphi(x)\frac{\partial u}{\partial \nu}(x)-u(x)\frac{\partial\varphi}{\partial \nu}(x)d\sigma\\
&\displaystyle =\int_{\{x\in\Omega|u(x)>0\}}-\varphi(x)\Delta u(x)dx+\int_{\partial\{x\in\Omega|u(x)>0\}}\varphi(x)\frac{\partial u}{\partial \nu}(x) d\sigma\\
&\displaystyle \leq\int_{\Omega}-\chi_{u>0}(x)\varphi(x)\Delta u(x) dx.
\end{array}
\end{equation}
Here we used the fact that $\frac{\partial u}{\partial \nu}\leq 0$, $u=0$ on $\partial\{u>0\}\cap\Omega$, and that $\varphi=0$, $\frac{\partial\varphi}{\partial\nu}=0$ on $\partial\Omega$.
\begin{equation}\label{lcklem1.4}
\begin{array}{rl}
&\displaystyle\int_{\Omega} w(x)
\big[\!-\!{\rm div}(\vec{b}(x)\varphi(x))+c(x)\varphi(x)\big]dx\\
=&\displaystyle\int_{\{x\in\Omega|u(x)>0\}} u(x)\big[\!-\!{\rm div}(\vec{b}(x)\varphi(x))+c(x)\varphi(x)\big]dx\\
=&\displaystyle \int_{\{x\in\Omega|u(x)>0\}}\big[\vec b(x)\cdot\nabla u(x)+c(x)u(x) \big]\varphi(x) dx\\
&\displaystyle-\int_{\partial\{x\in\Omega|u(x)>0\}}u(x)\varphi(x)\vec{b}(x)\cdot \nu(x) d\sigma\\
=&\displaystyle \int_{\Omega}\big[\vec b(x)\cdot\nabla u(x)+c(x)u(x) \big]\chi_{u>0}(x)\varphi(x) dx.
\end{array}
\end{equation}
Adding up (\ref{lcklem1.3}) and (\ref{lcklem1.4}):
\begin{equation}\label{lck528}
\int_{\Omega}w(x)[-\Delta\varphi(x)-{\rm div}(\vec{b}(x)\varphi(x))+c(x)\varphi(x)]dx\leq\int_{\Omega}\varphi(x)f(x)\chi_{u>0}(x)dx.
\end{equation}

Hence, $-\Delta w(x)+\vec b(x)\cdot\nabla w(x)+c(x)w(x)\leq f(x)\chi_{u>0}(x)$, in $\mathcal{D}'(\Omega)$.

\textit{Step 2.} In this step, we will drop the assumption that $\{x\in\Omega|u(x)>0\}$ has $C^1$-boundary.
Denote $u_\sigma=u-\sigma$ for $\sigma>0$.
By Sard's theorem, we can choose a sequence $\{\sigma_k\}\rightarrow 0 $ such that for each $k$,  the set $\{x\in\Omega|u_{\sigma_k}(x)>0\}$ has $C^1$-boundary.
Applying \eqref{lck528} to $w_{\sigma_k}=\max\{u_{\sigma_k}(x),0\}$, we have
\begin{equation}\label{lcklem1.6}
\int_{\Omega}w_{\sigma_k}(x)[-\Delta\varphi(x)-{\rm div}(\vec{b}(x)\varphi(x))+c(x)\varphi(x)]dx\leq\int_{\Omega}\varphi(x)(f(x)-\sigma_kc(x))\chi_{u>\sigma_k}(x)dx,
\end{equation}
where $\varphi(x)\in C_0^{\infty}(\Omega)$ is a nonnegative test function.
In (\ref{lcklem1.6}) we have
\begin{enumerate}
\item $w(x)-\sigma\leq w_\sigma(x)\leq w(x)$, hence $w_{\sigma_k}\rightarrow w$ in $L^1_{\rm loc}(\Omega)$, as $k\rightarrow\infty$;
\item $\displaystyle \bigcup_{k=1}^\infty\{x\in\Omega|u(x)>\sigma_k\}=\{x\in\Omega|u(x)>0\}$, where $\big\{\{x\in\Omega|u(x)>\sigma_k\}\big\}_{k=1}^\infty$ is a sequence of  monotonically increasing sets, hence $\chi_{u>\sigma_k}\rightarrow\chi_{u>0}$ in $L^1_{\rm loc}(\Omega)$, as $k\rightarrow\infty$;
\item $c(x)$ is bounded;
\item $\varphi$ has compact support.
\end{enumerate}

Letting $k\rightarrow\infty$, we get
\begin{equation}\label{lcklem1.7}
\int_{\mathbb{R}^n}w(x)[-\Delta \varphi(x)-{\rm div}(\vec{b}(x)\varphi(x))+c(x)\varphi(x)]dx\leq\int_{\mathbb{R}^n}\varphi(x)f(x)\chi_{u>0}(x)dx,
\end{equation}
i.e. $-\Delta w(x)+\vec b(x)\cdot\nabla w(x)+c(x)w(x)\leq f(x)\chi_{u>0}(x)$, in $\mathcal{D}'(\Omega)$.

\textit{Step 3.} Consider the general case that $v(x)\neq0$. Applying \textit{Step 2} to $u(x)-v(x)$, we derive
\begin{equation}\label{lcklem1.8}
-\Delta h(x)+\vec b(x)\cdot\nabla h(x)+c(x)h(x)\leq (f(x)-g(x))\chi_{u>v}(x)\ \ {\rm in}\ \mathcal{D}'(\Omega),
\end{equation}where $h(x)=\max\{u(x)-v(x),0\}$.
Adding (\ref{lcklem1.1})$_{\rm ii}$ to (\ref{lcklem1.8}) we derive
\begin{equation*}
-\Delta w(x)+\vec{b}(x)\cdot\nabla w(x)+c(x)w(x)\leq f(x)\chi_{u>v} + g(x)\chi_{u\leq v}\ \ {\rm in}\ \mathcal{D}'(\Omega).
\end{equation*}
This proves Lemma \ref{lcklem1}.  \end{proof}

\subsubsection{The proof of Lemma \ref{lcklem2}}
\begin{proof}We only focus on the statement that for $w(x)=\max\{u(x),0\}$, if $\partial\{x\in\Omega|u(x)>0\}$ is an $n-1$ dimensional manifold, then
\begin{equation}\label{3.1}
\int_{\mathbb{R}^n} w(x)(-\Delta)^s\varphi(x)dx\leq\int_{\{x\in\Omega|u(x)>0\}}(-\Delta)^{s}u(x)\varphi(x)dx,
\end{equation}
for any given $0\leq\varphi(x)\in C_0^\infty(\mathbb{R}^n)$, and then the rest of the proof follows exactly the same as  Lemma $\ref{lcklem1}$.

Denote $U=\{x\in\Omega|u(x)>0\}$. For a fixed $ \epsilon >0$, it holds
\begin{equation}\label{3.2}
\begin{array}{ll}
&\displaystyle\int_{\mathbb{R}^n}w(x)\int_{\mathbb{R}^n\backslash |x-y|\leq\epsilon}\frac{\varphi(x)-\varphi(y)}{|x-y|^{n+2s}}dydx
=\int_U \int_{\mathbb{R}^n\backslash |x-y|\leq\epsilon}u(x)\frac{\varphi(x)-\varphi(y)}{|x-y|^{n+2s}}dydx\\
=&\displaystyle\int_U\int_{\mathbb{R}^n\backslash |x-y|\leq\epsilon}\frac{\varphi(x)(u(x)-u(y))}{|x-y|^{n+2s}}dydx
+\int_U\int_{\mathbb{R}^n\backslash |x-y|\leq\epsilon}\frac{u(y)\varphi(x)-u(x)\varphi(y)}{|x-y|^{n+2s}}dydx\\
=&\displaystyle\int_U\int_{\mathbb{R}^n\backslash |x-y|\leq\epsilon}\frac{\varphi(x)(u(x)-u(y))}{|x-y|^{n+2s}}dydx\\
&\displaystyle+\underbrace{\int_U\int_{U\backslash |x-y|\leq\epsilon}\frac{u(y)\varphi(x)-u(x)\varphi(y)}{|x-y|^{n+2s}}dydx}_{=0,\ {\rm since\ changing\ the\ place\ of\ }x\ {\rm and}\ y,\ {\rm the\ value\ remains\ fixed}}\\
&+\displaystyle \int_U\int_{U^c\backslash |x-y|\leq\epsilon}\underbrace{\frac{u(y)\varphi(x)-u(x)\varphi(y)}{|x-y|^{n+2s}}}_{\leq0,\ {\rm since}\ u(y)\leq0\ {\rm and}\ u(x)>0}dydx\\
\leq &\displaystyle\int_U\varphi(x)\int_{\mathbb{R}^n\backslash |x-y|\leq\epsilon}\frac{u(x)-u(y)}{|x-y|^{n+2s}}dydx.
\end{array}
\end{equation}

Since  $\varphi(x)\in C_0^\infty(\Omega)$ and $u(x)$ is smooth, the above integrals are bounded. Then  letting $\epsilon\rightarrow0$, we arrive at (\ref{3.1}). The proof comes from \cite{li0}. We provide it here for the convenience of readers.
\end{proof}

\subsubsection{The proof of Theorem \ref{lckcor1}}

\begin{proof}
Without loss of generality, we may assume that $c(x)\equiv c_0$. For any domains $\Omega'$ and $\Omega''$ with $\Omega'\subset\subset\Omega$. Define $d=\frac{1}{2}d(\partial\Omega,\bar{\Omega'})>0$, and let $\delta\in(0,d)$. For any $x\in\Omega'$, denote $u^\delta(x)=J_\delta u(x)$, $v^\delta(x)=J_\delta v(x)$, $f^\delta(x)=J_\delta f(x)$, and $g^\delta(x)=J_\delta g(x)$. Then
\begin{equation}\label{lckcor1.2}
\left.
\begin{aligned}
&\displaystyle -\Delta u^\delta(x)+\vec{b}(x)\cdot\nabla u^\delta(x)+c_0u^\delta(x)\leq f^\delta(x)+M_\delta(x)\\
&\displaystyle  -\Delta v^\delta(x)+\vec{b}(x)\cdot\nabla v^\delta(x)+c_0v^\delta(x)\leq g^\delta(x)+N_\delta(x)
\end{aligned}\ \right\}\ {\rm in}\ \Omega',
\end{equation}
where $M_\delta$ and $N_\delta$ converges to $0$ in $L^1_{\rm loc}(\mathbb{R}^n)$ according to Lemma \ref{lcklem3}.

Applying Lemma \ref{lcklem1}, we get
\begin{equation}\label{lckcor1.3}
 -\Delta w^\delta(x)+\vec{b}(x)\cdot\nabla w^\delta(x)+c_0w^\delta(x)
\leq f^\delta(x)\chi_{u^\delta> v^\delta} + g^\delta(x)\chi_{u^\delta\leq v^\delta} +M_\delta(x)+N_\delta(x)\ \ {\rm in}\ \mathcal{D}'(\Omega'),
\end{equation}
where $w^\delta(x)=\max\{u^\delta(x),v^\delta(x)\}$. Choose a subsequence $\delta_k\rightarrow 0$ as $k\rightarrow\infty$. Since  $u^\delta(x)-v^\delta(x)$ converges to $u(x)-v(x)$ in $L^1(\Omega')$, we can apply the Egorov's theorem, $\forall\epsilon>0, \exists A_\epsilon\subset\Omega'\ (\mu(A_\epsilon)<\epsilon)$ and $u^{\delta_k}(x)-v^{\delta_k}(x)$ converges to $u(x)-v(x)$ uniformly on $\Omega' \backslash A_\epsilon$.

Fix $\sigma>0$, we have $1=\chi_{\{u-v>\sigma\}\backslash A_\epsilon}+\chi_{\{u-v<-\sigma\}\backslash A_\epsilon}+\chi_{\{|u-v|\leq\sigma\}\cup A_\epsilon}$. For $\delta$ sufficiently small, we calculate that in $\Omega'$:
\begin{equation}
\begin{array}{rl}
&f^{\delta_k}(x)\chi_{u^{\delta_k}> v^{\delta_k}} + g^{\delta_k}(x)\chi_{u^{\delta_k}\leq v^{\delta_k}}\\
=& (f^{\delta_k}(x)\chi_{u^{\delta_k}> v^{\delta_k}} + g^{\delta_k}(x)\chi_{u^{\delta_k}\leq v^{\delta_k}})(\chi_{\{u-v>\sigma\}\backslash A_\epsilon}+\chi_{\{u-v<-\sigma\}\backslash A_\epsilon}+\chi_{\{|u-v|\leq\sigma\}\cup A_\epsilon})\\
=&f^{\delta_k}(x)\chi_{\{u-v>\sigma\}\backslash A_\epsilon} + g^{\delta_k}(x)\chi_{\{u-v<-\sigma\}\backslash A_\epsilon}\\&+(f^{\delta_k}(x)\chi_{u^{\delta_k}> v^{\delta_k}} + g^{\delta_k}(x)\chi_{u^{\delta_k}\leq v^{\delta_k}})\chi_{\{|u-v|\leq\sigma\}\cup A_\epsilon}\\
\leq& f^{\delta_k}(x)\chi_{\{u-v>\sigma\}\backslash A_\epsilon}\! +\! g^{\delta_k}(x)\chi_{\{u-v<-\sigma\}\backslash A_\epsilon}+\max\{f^{\delta_k}(x),g^{\delta_k}(x)\}\chi_{\{|u-v|\leq\sigma\}\cup A_\epsilon}.
\end{array}
\end{equation}
Substituting this into \eqref{lckcor1.3}, one obtains in the sense of $\mathcal{D}'(\Omega')$:
\begin{equation}
\label{lckcor1.4}
\begin{array}{rl}
& -\Delta w^{\delta_k}(x)+\vec{b}(x)\cdot\nabla w^{\delta_k}(x)+c_0w^{\delta_k}(x)-M_{\delta_k}(x)-N_{\delta_k}(x)\\
\leq& f^{\delta_k}(x)\chi_{\{u-v>\sigma\}\backslash A_\epsilon}\! +\! g^{\delta_k}(x)\chi_{\{u-v<-\sigma\}\backslash A_\epsilon}+\max\{f^{\delta_k}(x),g^{\delta_k}(x)\}\chi_{\{|u-v|\leq\sigma\}\cup A_\epsilon}.
\end{array}
\end{equation}
Observing that as $k\rightarrow\infty$  $M_\delta$ and $N_\delta$ converges to $0$ in $L^1_{\rm loc}(\mathbb{R}^n)$ (Lemma \ref{lcklem3}), one derives:
\begin{align*}
&\left.
\begin{aligned}
&-\Delta w^{\delta_k}(x)+\vec{b}(x)\cdot\nabla w^{\delta_k}(x)
+c_0w^{\delta_k}(x)-(M_{\delta_k}(x)+N_{\delta_k}(x))& &\\
&\rightarrow -\Delta w(x)+\vec{b}(x)\cdot\nabla w(x)+c_0w(x)
\end{aligned}\right.
\ \ &{\rm in}\ \mathcal{D}'(\Omega')&;\\
{\rm and}\ 
&\left.
\begin{aligned}
&f^\delta(x)\chi_{\{u-v>\sigma\}\backslash A_\epsilon}+
g^\delta(x)\chi_{\{u-v<-\sigma\}\backslash A_\epsilon}
 +\max\{f^\delta(x),g^\delta(x)\}\chi_{\{|u-v|\leq\sigma\}\cup A_\epsilon}& &\\
&\rightarrow f(x)\chi_{\{u-v>\sigma\}\backslash A_\epsilon} + g(x)\chi_{\{u-v<-\sigma\}\backslash A_\epsilon}
+\max\{f(x),g(x)\}\chi_{\{|u-v|\leq\sigma\}\cup A_\epsilon}
\end{aligned}
\right.\ \ &{\rm in}\ \mathcal{D}'(\Omega')&.
 \end{align*}
Letting $k\rightarrow\infty$ in \eqref{lckcor1.4} we get
\begin{equation}\label{lckcor1.5}\left.
\begin{aligned}
& -\Delta w(x)+\vec{b}(x)\cdot\nabla w(x)+c_0w(x)\\
&\leq  f(x)\chi_{\{u-v>\sigma\}\backslash A_\epsilon} + g(x)\chi_{\{u-v<-\sigma\}\backslash A_\epsilon}
+\max\{f(x),g(x)\}\chi_{\{|u-v|\leq\sigma\}\cup A_\epsilon}
\end{aligned}\right.\quad {\rm in}\ \mathcal{D}'(\Omega').
\end{equation}

Then, let $\epsilon\rightarrow0^+$ in (\ref{lckcor1.5}):
\begin{equation}\label{lckcor1.6}
\left.
\begin{aligned}
&-\Delta w(x)+\vec{b}(x)\cdot\nabla w(x)+c_0w(x)\\
&\leq f(x)\chi_{\{u-v>\sigma\}} + g(x)\chi_{\{u-v<-\sigma\}}+\max\{f(x),g(x)\}\chi_{\{|u-v|\leq\sigma\}}
\end{aligned}\right.
\quad  {\rm in}\ \mathcal{D}'(\Omega').
\end{equation}

Finally,  let $\sigma\rightarrow0^+$ in (\ref{lckcor1.6}). Since it is clear that $\chi_{\{u-v>\sigma\}}$, $\chi_{\{u-v<-\sigma\}}$, and $\chi_{\{|u-v|\leq\sigma\}}$  converges to $\chi_{u>v}$, $\chi_{u<v}$, and $\chi_{u=v}$ in $L^1(\Omega')$, respectively, we obtain
\begin{equation*}
-\Delta w(x)+\vec{b}(x)\cdot\nabla w(x)+c_0w(x)\leq f(x)\chi_{u>v} + g(x)\chi_{u<v}+\max\{f(x),g(x)\}\chi_{u=v}\ \  {\rm in}\ \mathcal{D}'(\Omega').
\end{equation*}
Notice that $\Omega'$ is an arbitrary pre-compact domain in $\Omega$, we then arrive at the desired result.
\end{proof}

\subsubsection{The proof of Theorem \ref{lckcor2}}
\begin{proof}The proof is similar to the proof of Theorem \ref{lckcor1}.
\end{proof}

{\footnotesize REMARK} \textsc{3}: {\em The special but essential case of Lemma \ref{l 3.1} when $f(x)=g(x)=0$ has been proved in \cite{silvestre} under the assumptions $u(x)$ and $v(x)\in \mathcal{L}_{2s}$  are lower semi-continuous. Here we do not require lower semi-continuity.}

\subsection{Some other estimates for the proof of B\^{o}cher type theorem}

\quad\!\!\quad Here are two corollaries that are directly used in the proof of B\^{o}cher type theorem.
\begin{corollary}\label{lckcor3}
Assume that  $u$ is a smooth function satisfying
\begin{equation}
-\Delta u(x)+\vec b(x)\cdot\nabla u(x)+c_0u(x)=f(x)\ \ {\rm in}\ \mathbb{R}^n,
\end{equation}
where $\|\vec b(x)\|_{C^1(\mathbb{R}^n)}\leq c_0$ and $\{x\in\mathbb{R}^n|u(x)>0\}\subset\subset\mathbb{R}^n$.

Letting $w(x)=u^+(x)$, then for the mollified function $w^\delta(x)=J_\delta w(x)$, we have
\begin{equation}\label{lckcor3.0}
-\Delta w^\delta(x)+\vec b(x)\cdot\nabla w^\delta(x)+c_0w^\delta(x)\leq J_\delta(f\chi_{u>0})(x)+c_0\delta\|\nabla u\|_{L^\infty(\{u>0\})}\chi_{u>0}^\delta(x).
\end{equation}
\end{corollary}
\begin{corollary}\label{lckcor4}

Assume that  $u$ is a smooth function satisfying
\begin{equation}
(-\Delta)^su(x)+\vec b(x)\cdot\nabla u(x)+c_0u(x)=f(x)\ \ {\rm in}\ \mathbb{R}^n,
\end{equation}
where $\|\vec b(x)\|_{C^1(\mathbb{R}^n)}\leq c_0$  and $\Omega=\{x\in\mathbb{R}^n|u(x)>0\}\subset\subset\mathbb{R}^n$.

Letting $w(x)=u^+(x)$, then for the mollified function $w^\delta(x)=J_\delta w(x)$, we have
\begin{equation}\label{lckcor4.0}
(-\Delta)^sw^\delta(x)+\vec b(x)\cdot\nabla w^\delta(x)+c_0w^\delta(x)\leq J_\delta(f\chi_{u>0})(x)+c_0\delta\|\nabla u\|_{L^\infty(\{u>0\})}\chi_{u>0}^\delta(x).
\end{equation}
\end{corollary}
These two statements are direct corollaries of Lemma \ref{lcklem1}, Lemma \ref{lcklem2} and Lemma \ref{lcklem4}.

\end{document}